\documentclass[11pt]{amsart}

\usepackage{amssymb}

\usepackage[varg]{txfonts}

\newtheorem{thm}{Theorem}[section]

\newtheorem{lem}[thm]{Lemma}

\theoremstyle{definition}

\theoremstyle{remark}

\newcommand\al{\alpha}
\newcommand\bt{\beta}
\newcommand\Gm{\Gamma}
\newcommand\Gf{\Gamma} 
\newcommand\gm{\gamma}
\newcommand\Dt{\Delta}
\newcommand\dt{\delta}
\newcommand\e{\varepsilon}

\renewcommand\th{\vartheta}
\renewcommand\k{\kappa}
\newcommand\Ld{\Lambda}
\newcommand\ld{\lambda}

\newcommand\x{\xi}
\newcommand\p{\pi}
\renewcommand\r{\rho}

\newcommand\ups{\upsilon}

\newcommand\ch{\chi}
\newcommand\ps{\psi}

\newcommand\om{\omega}

\newcommand\CC{{\mathbb{C}}}
\newcommand\NN{{\mathbb{N}}}
\newcommand\QQ{{\mathbb{Q}}}
\newcommand\RR{{\mathbb{R}}}
\newcommand\ZZ{{\mathbb{Z}}}

\newcommand\Ecal{{\mathcal{E}}}
\newcommand\Hcal{{\mathcal{H}}}
\newcommand\Kcal{{\mathcal{K}}}
\newcommand\Lcal{{\mathcal{L}}}
\newcommand\Ocal{{\mathcal{O}}}

\newcommand\y{{\mathfrak y}}

\newcommand\re{\operatorname{Re}}
\newcommand\im{\operatorname{Im}}
\newcommand\isdef{\mathrel{:\mskip1mu=}}
\newcommand\oh{\mathrm{O}}
\newcommand\fd{{\mathfrak{F}}}

\newcommand\vol{\mathrm{vol}}
\newcommand\siegel{{\mathfrak S}}
\newcommand\cnt{{\mathbf N}}
\newcommand\onebold{{\mathbf 1}}
\newcommand\vmin{V_{\mathrm{min}}}

\newcommand\hypgeom{{}_2\!F_{\!1}}
\newcommand\expl{{\mathrm{expl}}}

\newcommand\SL{\mathrm{SL}}
\newcommand\PSL{\mathrm{PSL}}
\newcommand\PSO{\mathrm{PSO}}

\newcommand\dist{\mathrm{d}}
\newcommand\supp{\mathrm{Supp}}

\newcommand\uhp{\mathfrak{H}}

\renewcommand\setminus{\smallsetminus}

\renewcommand\={\;=\;}

\makeatletter
\newcommand\matc[4]{\left[ {#1\@@atop #3}{#2\@@atop #4}\right]}
\newcommand\matr[4]{\left[ {\hfill #1\@@atop\hfill #3}{\hfill
#2\@@atop\hfill #4}\right]}
\newcommand\rmatc[4]{\left( {#1\@@atop #3}{#2\@@atop #4}\right)}
\newcommand\rmatr[4]{\left( {\hfill #1\@@atop\hfill #3}{\hfill
#2\@@atop\hfill #4}\right)}
\makeatother

\newcommand\txtfrac[2]{{\textstyle\frac{#1}{#2}}}

\begin{document}

\title{New lattice point 
asymptotics for products of upper half planes}

\author{R.W.\,Bruggeman}
\address{Mathematisch Instituut Universiteit Utrecht, Postbus 80010,
3508 TA Utrecht, Netherlands} \email{r.w.bruggeman@uu.nl}

\author{F.\,Grunewald}
\address{Mathematisches Institut der Heinrich-Heine-Universit\"at
D\"usseldorf, Uni\-ver\-si\-t\"ats\-stra\ss{e}~1, 40225 D\"usseldorf,
Germany}
\email{fritz@math.uni-duesseldorf.de}

\author{R.J.\,Miatello}
\address{FaMAF-CIEM, Universidad Nacional de C\'or\-do\-ba,
C\'or\-do\-ba~5000, Argentina} \email{miatello@mate.uncor.edu}



\begin{abstract}
Let $\Gamma$ be an irreducible lattice in $\PSL_2(\RR)^d$ ($d\in\NN$) and
$z$ a point in the $d$-fold direct product of the upper half plane.
 We study the discrete set of 
componentwise distances ${\bf D}(\Gm,z)\subset \RR^d$ defined in (1).
We prove asymptotic results on the number of $\gm\in\Gm$ such that
$d(z,\gamma z$ is contained in strips expanding in some directions
 and also in  expanding hypercubes. 
The results on the counting in expanding strips are new. The results on 
expanding hypercubes 
improve the existing error terms (\cite{GN10})  and generalize the Selberg
 error term for $d=1$.

We give an asymptotic formula for
the number of lattice points $\gamma z$ such that the hyperbolic
distance in each of the factors satisfies $d((\gamma z)_j, z_j)\le T$. 
The error term, as $T \rightarrow \infty$
generalizes the error term given by Selberg for $d=1$, also
we describe how the counting function depends on
$z$. We also prove asymptotic results when
the distance satisfies $A_j \le d((\gamma z)_j, z_j) < B_j$, with
fixed $A_j <   B_j$ in some factors, while in the remaining factors 
$0 \le   d((\gamma z)_j, z_j) \le T$ is satisfied.
\end{abstract}

\maketitle

\tableofcontents

\section{Introduction} 

Let $\uhp=\{\, x+iy\in \CC : y> 0\,\}$ be the upper halfplane 
equipped with the hyperbolic metric $d:\uhp\times\uhp\to \RR$ and
its invariant measure induced by 
$\frac{dx\, dy}{y^2}$.
The group of orientation preserving isometries of this metric space is 
$\PSL_2(\RR)$. 
Let now $d$ be a natural number and consider the semisimple Lie group 
$\PSL_2(\RR)^d$ as acting on its corresponding symmetric space $\uhp^d$. 
We write $z=(z_1,\ldots,z_d)$ for the coordinates 
$z_1,\ldots,z_d\in\uhp$ of a point $z\in \uhp^d$. Let
us consider the vector valued distance function
\begin{equation}
d(z,u):=(d(z_1,u_1),\ldots ,d(z_d,u_d))\in \RR^d
\end{equation}
for points $z=(z_1,\ldots,z_d)$, $u=(u_1,\ldots,u_d)\in \uhp^d$.
The canonical invariant distance of $z,\, u$ is then the euclidean norm 
of $d(z,u)$. But other choices of norms (like the maximum norm) also induce
$\PSL_2(\RR)^d$-invariant metrics on $\uhp^d$. 

Let $\Gm \subset\PSL_2(\RR)^d$ be an irreducible lattice.
A lattice in $\PSL_2(\RR)^d$ is a 
discrete subgroup $\Gm \subset\SL_2(\RR)^d$ of finite
covolume, that is, the volume of the quotient
$\Gm\backslash \uhp^d$
in the canonical measure is finite.  
The lattice $\Gm \subset\SL_2(\RR)^d$
is called irreducible 
if all projections of $\Gm$ to non-trivial
subproducts of $\PSL_2(\RR)^d$ are dense.
A main example is the
\emph{Hilbert modular group} $\PSL_2(\Ocal)$ for the ring of integers
$\Ocal$ of a totally real number field $F$ of degree $d$ over~$\QQ$,
embedded in the product $\PSL_2(\RR)^d$ by the $d$ embeddings of $F$
into~$\RR$. The embedded group $\PSL_2(\Ocal)$ and all its subgroups of finite
index are irreducible lattices in $\PSL_2(\RR)^d$. They are not cocompact, 
which means that the quotient $\Gm\backslash \uhp^d$ is not compact.
In case $d\ge 2$
every irreducible lattice in $\PSL_2(\RR)^d$ which is not cocompact contains
a subgroup of finite index  which is $\PSL_2(\RR)^d$-conjugate to a subgroup 
of finite index in one of the
$\PSL_2(\Ocal)$.
Irreducible cocompact lattices in
$\PSL_2(\RR)^d$ are constructed from quaternion algebras over 
totally real number fields $F$. In case $d\ge 2$ these are 
up to conjugacy the only examples, by Margulis' 
arithmeticity theorem. See Section \ref{sect-Lgdsg} for more details.

Let $\Gm \subset\PSL_2(\RR)^d$ be an irreducible lattice and let 
$z\in \uhp^d$ be fixed. Consider the set of vector valued distances 
\begin{equation}\label{ddd}
{\bf D}(\Gm,z):=\{\, d(z,\gamma z)\in \RR^d\ :\ 
\gamma\in \Gm\, \}.
\end{equation}
This clearly is an infinite discrete subset of 
$\RR^d$. But what more can be said? In this paper we shall 
prove results which describe 
the distribution of the points of ${\bf D}(\Gm,z)$ in various regions like
strips or expanding polyhedra in $\RR^d$.

To give a precise 
formulation of our main results, we need to discuss some aspects of the
spectral theory of $L^2(\Gm\backslash\uhp^d)$. 
This Hilbert space has
an infinite dimensional subspace
$L^{2,\mathrm{discr}}(\Gm\backslash \uhp^d)$ with an orthonormal
basis $\{\ps_\ell\}$\,  ($\ell\in \NN\cup \{0\}$) 
of joint eigenvectors of the Laplace operators
$\Dt_j = -y_j^2 \partial_{x_j}^2-y_j^2\partial_{y_j}^2$ ($j=1,\ldots d$)
in the 
factors. Among the eigenfunctions is the constant function
$\ps_0(z)=\bigl(\vol(\Gm\backslash\uhp^d)\bigr)^{-1/2}$ for which the
eigenvalues of all $\Dt_j$ are all equal to $0$. The corresponding
multi-eigenvalues $\ld_\ell$ have finite multiplicities and form a
discrete set in $[0,\infty)^d$. For $\ell\geq 1$ one knows that
$\ld_{\ell,j}>0$ for all\, $j=1,\ldots, d$. If $\Gm$ is cocompact, then
$L^{2,\mathrm{discr}}(\Gm\backslash \uhp^d)$ is all of
$L^2(\Gm\backslash\uhp^d)$. Otherwise the elements of the orthogonal
complement of $L^{2,\mathrm{discr}}(\Gm\backslash \uhp^d)$ can be
described as sums of integrals of Eisenstein series.

We call a multi-eigenvalue $\ld_\ell$ \emph{exceptional} if
$0<\ld_{\ell,j}<\frac14$ for some coordinate~$j\in \{1,\ldots, d\}$. 
If $d\geq 2$, there
may be infinitely many exceptional eigenvalues, since there is no
bound on the other coordinates. If $0<\ld_{\ell,j}<\frac14$ for
all~$j$  we call $\ld_\ell$ \emph{totally exceptional}. There are at
most finitely many totally exceptional eigenvalues.

For a further discussion we use the parametrization
$\ld = \frac14-\tau^2$ by the \emph{spectral parameter}~$\tau$. In
$L^2(\Gm\backslash \uhp^d)$ all eigenvalues of local Laplace
operators are in $[0,\infty)$, so we can choose $\tau\in
i[0,\infty)\cup\bigl[0,\frac12\bigr]$. Thus we have
$\tau_{0,j}=\frac12$ for all~$j$, and $\re\tau_{\ell,j}<\frac12$ for
all~$j$.

For a congruence subgroup $\Gm$ of a Hilbert modular group it has been
shown by Kim and Shahidi, \cite{KS2}, that
$\re \tau_{\ell,j} \leq \frac19$ for all $\ell\geq 1$ and all~$j$.
For this situation a conjecture called after Selberg says that $\re
\tau_{\ell,j}=0$ for all $\ell\geq 1$ and all~$j$. Below we will
discuss other results concerning $\re\tau_{\ell,j}$, $\ell\geq 1$.
For the formulation of our results we summarize the information
concerning exceptional eigenvalues in the quantity
\begin{equation}\label{tau-hat} \hat\tau \= \hat\tau(\Gm)\;\isdef\;
\sup_{\ell\geq 1\,,\;1\leq j\leq d} \re\tau_{\ell,j}\,.
\end{equation}
This, by definition, is 
an element of $\bigl[0,\frac12\bigr]$. Since there may be
infinitely many $\ld_\ell$, 
the value $\frac12$ might occur in (\ref{tau-hat}), 
although here
we omit $\ell=0$.
 If $\hat\tau<\frac12$ one says
that $\Gm\backslash\uhp^d$ has a \emph{strong spectral gap}. 
In our later arguments we use that $\hat\tau<\frac12$, a fact 
proved in \cite{KeSa} by Kelmer and Sarnak 
for cocompact~$\Gm$. 
If $\Gm$
is not cocompact and $d\geq 2$ then $\Gm$ contains a subgroup of finite
index which is conjugate to a congruence subgroup  
of a Hilbert modular group $\SL_2(\Ocal)$, for which the
results of Kim and Shahidi, \cite{KS2}, imply that
$\hat\tau\leq \frac19$. Here we use the important fact, proved in 
\cite{Serre} that every subgroup of finite index in $\SL_2(\Ocal)$ 
is a congruence subgroup.

Let now $E\subset \{1,\ldots,d\}$ be a non-empty subset
and let $I:=\{\, I_j \ : \ j\in E\,\}$ be a set of 
bounded intervals
$I_j:=[A_j,B_j) \subset [0,\infty)$.
Define for $T >0$ 
\begin{equation} 
{S}(E,I;T):=\{\, (x_1,\ldots x_d)\in \RR^d \ : \ x_j \in I_j\ {\rm for}\
j\in E, \ 0\le x_j\le T\ {\rm for } j\notin E\, \}.
\end{equation} 
We think of ${S}(E,I;T)$ as a strip of increasing height $T$ in $\RR^d$.
Given $z\in \uhp^d$ we introduce the counting quantity 
\begin{equation}\label{NEdef}
N_E(z;T):= \#\, \{\, \gm\in \Gm \ : \ d(z,\gamma z) \in {S}(E,I;T)\, \}.
\end{equation}
We show
\begin{thm}\label{thm-QEai}
Let $\Gm$ be an irreducible lattice in
$\PSL_2(\RR)^d$, with $d\geq 2$.
Let $E\subset \{1,\ldots,d\}$ be a subset with $e:=\#E\geq 1$.
Define $Q:= \{1,\ldots,d\}\setminus E$ and assume
$q:=\#Q \geq 1$. 
Let finite intervals
$[A_j,B_j)\subset[0,\infty)$ be given for $j\in E$, the quantity $N_E(z;T)$ in
\eqref{NEdef} satisfies as $T\rightarrow\infty$:
\begin{align*}
N_E(z;T) &\= \frac{\p^d \, 2^e}{\vol(\Gm\backslash\uhp^d)} \;e^{q T}
\;\prod_{j\in E} \bigl( \cosh B_j - \cosh A_j\bigr)
\\
&\qquad\hbox{}
+
\begin{cases}
\displaystyle \oh_{\Gm,E} \biggl( n(z)\, \exp\biggl( \frac{d+1}{d+2}\,
q T \biggr) \biggr)
& \displaystyle\text{ if }\hat\tau\leq \frac q{2(d+2)}\,,\medskip\\
\displaystyle \oh_{\Gm,E}\biggl(n(z)\, \exp\biggl(
\frac{1+2\hat\tau+e}{2+e}\, q T\biggr)
\biggr)& \displaystyle\text{ if }\frac q{2(d+2)}\leq
\hat\tau<\frac12\,.
\end{cases}
\end{align*}
\end{thm}
This is a specialization of Theorem~\ref{thm-QEa}, where we allow
somewhat more general conditions on the
components of the vector valued distances. The $E$ in $\oh_{\Gm,E}$ implies
an implicit dependence on the intervals $[A_j,B_j)$ with $j\in E$.
In Theorem~\ref{thm-QEai} we have 
$d\geq 2$ and $1\leq q=\#Q \leq d-1$. As explained above $\hat\tau<\frac12$
holds and
it depends on the relative
sizes of $d$ and $q$ which of the error terms is applicable.

We shall describe now another result pertaining to the more standard lattice 
point problems.
We consider the asymptotic distribution of the 
orbit points $\gm z$ ($\gamma\in \Gamma$)
for a given point
$z\in X$ and a discontinuous group
of motions $\Gm$ acting on a symmetric space $X$.
 In the case when $X=\uhp$ is the
upper half plane many authors have contributed to this problem, for
instance \cite{Hu59}, \cite{Pa75} and~\cite{Iw95}. The best result 
concerning error terms 
is due to Selberg (see the Bombay and G\"ottingen lectures in 
\cite{Selcol}). It gives
\begin{equation}
\begin{aligned}
\# \bigl\{ \gm\in \Gm \;&:\; \dist(\gm z,z) \leq T\bigr\} \=
\frac\pi{\vol(\Gm\backslash \uhp)} \, e^T \\
&\qquad\hbox{}
+ \sum_\ell
\pi^{1/2}\,|\ps_\ell(z)|^2\,\frac{\Gf(\tau_\ell)}{\Gf(\tau_\ell+3/2)}
\, e^{(1/2+\tau_\ell)T}
+ \oh \bigl( e^{\frac 23 T} \bigr)\quad(T\rightarrow\infty)\,.
\end{aligned}
\end{equation}
 The functions $\ps_j$ form an finite
orthonormal system (possibly empty)
of eigenfunctions with eigenvalue $\frac14-\tau_\ell^2$ of the
hyperbolic Laplace operator acting on $L^2(\Gm\backslash \uhp)$
with $0<\tau_\ell<\frac14$. This result holds
for all cofinite 
discrete subgroups of $\PSL_2(\RR)$, cocompact or
not.

Let $\Gm$ now be an irreducible lattice in
$\PSL_2(\RR)^d$, with $d\in \NN$. 
For $z\in \uhp^d$ we define
\begin{equation}
N(z;T):= \#\, \left( {\bf D}(\Gm,z)\cap \{\, x\in \RR^d \ : \ 
\max(x)\le T\, \}\right)
\end{equation}
where $\max(x)$ is the maximum of the absolute values of the coordinates of
the vector $x\in \RR^d$ and ${\bf D}(\Gm,z)$ is defined in (\ref{ddd}). 
We show
\begin{thm}\label{thm-Qai}
Let $\Gm$ be an irreducible lattice in
$\PSL_2(\RR)^d$, with $d\in \NN$ and let $z\in \uhp^d$ be given.
With
$\hat\tau=\hat\tau(\Gm)$ as in~\eqref{tau-hat}, and with the quantity
$n(z)$ as defined in~\eqref{ndef}, the counting function $N(z;T)$ has
the following asymptotic behavior as $T\rightarrow\infty$:
\begin{itemize}
\item
If $0\leq \hat\tau(\Gm) \leq \frac d{2(d+2)}$ \emph{(large spectral
gap)}, then
\[ N(z;T) \= \frac{\pi^d}{\vol(\Gm\backslash\uhp^d)}\, e^{dT} + \oh
_\Gm \biggl( n(z)\exp\biggl( \frac{d+1}{d+2}dT\biggr)
\biggr)\,.\]
\item If $\frac d{2(d+2)}\leq \hat\tau(\Gm)  \leq \frac12$ \emph{(small
spectral gap)}, then
\begin{align*}
N(z;T) &\= \frac {\pi^d}{\vol(\Gm\backslash\uhp^d)} e^{dT}
+ \sum_{\ell\geq1\,,\; \forall_j\; \tau_{\ell,j}\in(0,1/2)}
\bigl|\ps_\ell(z)\bigr|^2 \prod_{j=1}^d\biggl( \frac{\sqrt \pi\,
\Gf(\tau_{\ell,j})}{\Gf(3/2+\tau_{\ell,j})} e^{(1/2+\tau_{\ell,j})T}
\biggr)\\
&\qquad\hbox{}
+ \oh_\Gm\biggl( n(z) \exp\biggl( \frac{2d+2(d-1)\hat\tau}3 T \biggr)
\biggr)
\,.
\end{align*}
\end{itemize}
\end{thm}
Theorem~\ref{thm-Qai} is a special case of Theorem~\ref{thm-Qa}, where
we allow a more general counting quantity than $N(z;T)$. The function
$n(z)$ in~\eqref{ndef} is positive on $\Gamma\backslash\uhp^d$ and
grows when $z$ approaches a cusp.

We note that we do not try to put one distance function on $\uhp^d$,
but work with the vector of the distances in the factors. Partly,
this is because in this way it is easier to apply the spectral
theory. Partly, it reflects the fact that there is not one distance
function on $\uhp^d$ that is preserved by the action of
$\PSL_2(\RR)^d$, but infinitely many. 

For a small spectral gap totally exceptional eigenfunctions appear
explicitly in the asymptotic estimate. Of course, some of these
exceptional contributions may happen to be absorbed by the error
term. For large spectral gaps, further improvement of our knowledge
of $\hat\tau(\Gm)$ does not improve the quality of the error term in
our asymptotic formula. This holds in particular for the congruence
case, in which we know that $\hat\tau\leq \frac19$. The main term is
always larger than the error term, even if we would have
$\hat\tau(\Gm)=\frac12$
(no spectral gap).

The case $d=1$ in Theorem~\ref{thm-Qai} concerns lattice point
counting for groups acting on the upper half plane. The best known
error term $\oh(e^{\frac 23 T})$ coincides with the error term in
Theorem~\ref{thm-Qai} for $d=1$. The papers \cite{Gu80}, \cite{Ba82},
\cite{LP82}, \cite{Le87} and \cite{BMW99} treat lattice point
counting for other symmetric spaces of rank one. The situation in
this paper, with rank~$d$, falls within the scope of \cite{GN7}
and~\cite{GN10}, in which Gorodnik and Nevo consider counting of
lattice points over quite general families of sets in quotients of
more general Lie groups. Their error terms for $\Gm\backslash\uhp^d$
are weaker than those in Theorem~\ref{thm-Qai}. They get 
$\oh\bigl(\exp\bigl( \frac 56 T\bigr)
\bigr)$ in the case $d=1$ and $\Gm$ not cocompact, and 
$\oh \big( \exp\bigl( \frac{4d+1}{4d+2}dT \bigr)
\bigr)$ for the Hilbert modular case and $d\ge 2$, 
which should be compared with Selberg's bound
$\oh \bigl( \frac 23 T\bigr)\bigr)$ for $d=1$, and with
$\oh \bigl( \exp\bigl( \frac{d+1}{d+2}dT \bigr)\bigr)$ in
Theorem~\ref{thm-Qai} for general~$d$. We emphasize that the class of
counting problems considered by Gorodnik and Nevo is much larger than
ours. They consider quite general families $t\mapsto G_t$ of regions
in much more general groups than $\PSL_2(\RR)^d$, that have to grow
in all directions. They use an ergodic method that can be applied in
all these cases, without using more spectral information than the
size of the spectral gap. The counting in Theorem~\ref{thm-QEai} is
over regions that are constant in some coordinate directions, and
hence do not satisfy the conditions in~\cite{GN7}.

In the proofs we apply the spectral theory of automorphic forms. We
give the main proof in \S\ref{sect-pfs}. The approach is sketched in
the introduction to~\S\ref{sect-st-se} and in Subsections
\ref{sect-mp} and~\ref{sect-ap}. The idea is that the sums $N(z;T)$
and $N_E(z;T)$ are replaced by smooth approximations. This smoothness
ensures that the new quantities have a spectral decomposition that
converges pointwise. In this spectral expansion we single out the
terms corresponding to the constant functions and to totally
exceptional eigenfunctions, if these are present. These give the main
terms in the asymptotic expansion. The remaining part of the
spectral decomposition is estimated 
using
the estimate of the spectral measure in Theorem~\ref{thm-spm}.
We use an approach similar to one
in~\cite{Iw95} that makes explicit the dependence on the
point~$z\in\uhp^d$.

For the handling of the spectral decomposition the Selberg transform
discussed in~\S\ref{sect-St} is essential. In \S\ref{sect-eSt} we
prove the estimates and other facts that we need. In the proof of the
main theorems we also use some estimates of the counting function
(Lemmas \ref{lem-apl} and~\ref{lem-ape}) obtained without the use of
spectral theory. 
A main role in the proof is played by an estimate of
the spectral function given in Theorem~\ref{thm-spm}, proved
in~\S\ref{sect-spm}.

\section{Lie groups and discrete subgroups}\label{sect-Lgdsg}
Let $G$ be the Lie group $\PSL_2(\RR)^d$ for some integer $d\geq 1$.
The group $G$ acts on the product $\uhp^d$ of upper half planes by
fractional linear transformations in each factor. We will use the
letter~$j$ to index these factors. $G$ leaves invariant the vector
valued distance function
\begin{equation}
\dist(z,w) \= \bigl( \dist_j(z_j,w_j) \bigr)_{j\in \{1,\ldots,d\}}\,,
\end{equation}
where $\dist_j$ is the hyperbolic distance in the $j$-th factor. By
$\matc abcd$ we denote the class in $G$ represented by
$\rmatc abcd\in \SL_2(\RR)$.

We consider an irreducible lattice $\Gm\subset G$, as described in
Definition~5.20 and Corollary~5.21 of~\cite{Ra72}. So, for each of
the genuine subproducts $H$ of $\PSL_2(\RR)^d$ the projection
$\Gm\rightarrow H$ has dense image. In particular
$\Gm\backslash \uhp^d$ has finite volume, and the projection to each
of the factors is injective on~$\Gm$. (See also Corollary~5.23 in
{\sl loc.\ cit.})

Hilbert modular groups $\PSL_2(\Ocal)$ and their subgroups of finite index,
mentioned in the introduction, are examples. Cases for which
$\Gm\backslash G$ is compact can be derived from  quaternion algebras
$\Hcal$ over a totally real number field $F$ for which there is a 
non-empty set
$S$ of infinite places $j$ for which the tensor product
$F_{\!j} \otimes _F \Hcal$, with the completion $F_{\!j}$, is a
division algebra. Suppose that $\#S=d>0$. Let $\Hcal_\Ocal$ be an
order in $\Hcal$. Then the elements of reduced norm~$1$ in
$\Hcal_\Ocal$ have as their image in $\prod_{j\notin S} \PSL_2(F_{\!j})$
a cocompact discrete subgroup satisfying the assumptions above.

In the case $d=1$ most of the subgroups with finite index in
$\PSL_2(\ZZ)$ are not the image of a congruence subgroup of
$\SL_2(\ZZ)$. Moreover there many are irreducible discrete subgroups
of $\PSL_2(\RR)$ that are not commensurable to $\PSL_2(\ZZ)$. For
$d\geq 2$, Margulis has shown that all irreducible discrete subgroups
of $\PSL_2(\RR)^d$ are arithmetic, i.e., commensurable to a Hilbert
modular group or to a unit group of a quaternion algebra. (See
Theorem (1.11) in Chap.~IX of~\cite{Ma91}, or the discussion in \S7
of~\cite{Se70}.) Serre, \cite{Serre}, has shown that all subgroups
of finite index in $\SL_2(\Ocal)$ are congruence subgroups. So all
non-cocompact irreducible lattices contain a conjugate of a
congruence subgroup as a subgroup with finite index.

\section{A priori estimates}\label{sect-ape}
{}From here on we follow the usual practice of not working directly
with the hyperbolic distance $\dist$ on~$\uhp$, but with
\begin{equation}\label{u-d}
\begin{aligned}
u(z,z') &\= \frac{|z-z'|^2}{4yy'} \= \bigl( \sinh \txtfrac12
\dist(z,z')
\bigr)^2\,,\\
\dist(z,z')&\= 2\log\biggl( \sqrt{u(z,z')}+\sqrt{u(z,z')+1}\biggr)\,.
\end{aligned}
\end{equation}
For $U,V\in [0,\infty)^d$ such that $U_j<V_j$ for all~$j$, we consider
the counting quantity
\begin{equation}\label{cntdef}
\cnt(U,V;z) \= \#\bigl\{ \gm\in \Gm\;:\; U_j\leq u\bigl((\gm
z)_j,z_j\bigr)<V_j\text{ for all }j\bigr\}\,.
\end{equation}
To relate this to the quantities $N(z;T)$ and $N_E(z;T)$ used in the
introduction we will use that $\dist\downarrow 0$ corresponds to
$u \downarrow 0$ in such a way that
\begin{equation}
u\=\frac{\dist^2}4 + \oh(\dist^4)\,,\qquad \dist\= 2\sqrt u +
\oh(u^{3/2})\,,
\end{equation}
and that $u\rightarrow\infty$ corresponds to $\dist\rightarrow\infty$
in such a way that
\begin{equation}
u\= \frac14 e^\dist + \oh(1)\,,
\qquad \dist\= \log u + \log 4 + \oh(u^{-1})\,.
\end{equation}
\medskip

We will need a starting point for the estimation of $\cnt(U,V;z)$. We
give two estimates for the counting function. The first is based on a
simple volume argument. The second is important for the dependence of
our results on the geometry of $\Gm\backslash \uhp^d$. As
$z\in\uhp^d$ approaches a cusp there are more and more $\gm\in \Gm$
for which $\gm z $ is near~$z$.\smallskip

\begin{lem}\label{lem-apl}For $z\in \uhp^d$ and $U,V\in [0,\infty)^d$
such that $U_j<V_j$ for all~$j$, we have
\[ \cnt(U,V;z) \ll_{\Gm,z} \prod_j (V_j-U_j+1)\,.\]
\end{lem}
\begin{proof}For $w\in \uhp^d$ and $\dt>0$ we put
\[ B(w,\dt) \= \bigl\{ v\in \uhp^d\;:\; \forall_j\;
u(w_j,v_j)<\dt\bigr\}\,.\]

Let $z\in \uhp^d$ be given. The subgroup $\Gm_z$ of $\Gm$ fixing~$z$
is finite. See, {\sl e.g.}, Remark~2.14 in~\cite{Fr80}. By the
discontinuity of the action there is $\dt>0$ such that
$B(z,\dt) \cap B(\gm z,\dt)=\emptyset$ for all
$\gm\in \Gm\setminus\Gm_z$. The $\Gm$-invariance of $u$ implies that
$B(\gm_1,z,\dt) \cap B(\gm_2 z,\dt)=\emptyset$ for all
$\gm_1,\gm_2\in \Gm$ for which $\gm_1 z \neq \gm_2 z$.

For $P,Q\in [0,\infty)^d$ denote by $A(P,Q)$ the multi-annulus
\[ \bigl\{ v\in\uhp^d\;:\; \forall_j\; P_j\leq u\bigl( v_j,z_j
\bigr)<Q_j\bigr\}\,.\]
The pairwise disjoint sets $B(\gm z,\dt)$ with
$\gm z \in \Gm z \cap A(U,V)$ are contained in a slightly larger
multi-annulus $A\bigl( U(\dt) ,V(\dt) \bigr)$ with
$U(\dt)_j = U_j-\oh(1)$ and $V(\dt)_j = V_j+\oh(1)$. See~\eqref{u-d}.
Thus we have
\[ \#\bigl( \Gm z \cap A(U,V) \bigr) \;\leq \; \vol\bigl(
A(U(\dt),V(\dt)) \bigr) / \vol(B(z,\dt)) \ll_\dt
\prod_j(V_j-U_j+1)\,.\](The volume computation is easiest in a
distance coordinate $u=u(z,i)$ and an angular coordinate $\phi$. Then
$d\mu$ on $\uhp$ is given by $4 \, du\, d\phi$. See
(1.17)
in~\cite{Iw95}.)
Since
$\cnt(U,V;z) = \#\Gm_z \,\cdot\, \#\bigl( \Gm z\cap A(U,V) \bigr)$,
this proves the lemma.
\end{proof}

Next we aim at an estimate of $\cnt(z,w;0,V)$ when all $V_j$ are
small. In this estimate, the dependence on $z$ will be explicit. In
order to do this, we take into account some facts concerning the
geometry of the action of~$\Gm$ on~$\uhp^d$. (The approach is
motivated by that in the proof of Corollary~2.12 in~\cite{Iw95}.)

The assumptions on $\Gm$ imply that we can find a fundamental domain
$\fd$ that is compact in the case of cocompact $\Gm$, and is
contained in a union of Siegel domains otherwise:
\begin{align*} \fd &\subset \bigcup_\k g_\k \siegel_\k\,,
\displaybreak[0]
\\
\siegel(X_\k,Y_\k,V_\k) &\= \left\{z\;:\; x_j\in[-X_\k,X_\k]\text{ for
all }j\,,\; y_1y_2\cdots y_d\geq Y_\k\,,\right.\\
&\qquad\qquad\left. V_\k^{-1} \leq \frac {y_j}{y_{j+1}} \leq
V_\k\text{ for }1\leq j \leq d-1 \right\}\,\end{align*}
where $\k$ runs through a finite set of representatives
$\k=g_\k\infty$ of the $\Gm$-classes of cusps, with $g_\k \in G$,
$X_\k, Y_\k>0$ and $V_\k>1$. (See \cite{Fr80}, Chap.~I, \S2.)
Enlarging $Y_\k$ decreases $\siegel(X_\k,Y_\k,V_\k)$. There exists
$A>0$ such that the $g_\k \siegel_\k(X_\k,A,V_\k)$ are disjoint. For
each $B\geq A$ there is a compact set $C_B$ such that
\begin{equation}
\label{fddc}
\fd \subset C_B \cup \bigcup_\k g_\k
\siegel(X_\k,B,V_\k)\,.\end{equation}

We fix a fundamental domain and a disjoint decomposition of it induced
by~\eqref{fddc}, and define $\Gm$-invariant functions
$\y_1,\ldots,\y_d$ on $\uhp$ determined by the requirement that
for~$z\in \fd$:
\begin{equation}
\y_j(z) \= \begin{cases}
1&\text{ if $\Gm$ is cocompact, or if $z\in C_A$}\,,\\
\im(g_{\k,j}^{-1}\,z_j)&\text{ if }z\in g_\k \siegel(X_\k,A,V_\k)\,.
\end{cases}
\end{equation}
The product of the $\y_j(z)$ measures how far up in a cusp sector the
point $z$ is situated. Note that the $\y_j$ may be discontinuous, but
are bounded away from~$0$.

For $T \in (0,\infty)^d$ we put
\begin{equation}\label{nTdef}
n_j(T_{\!j},z) \= \max(1,\y_j(z)/T_{\!j})\,,\qquad n(T,z) \= \prod_j
n_j(T_{\!j},z)\,.
\end{equation}
We take
\begin{equation}
\label{ndef}n(z)\=n(\onebold,z) \= \prod_j \max\left( 1,\y_j(z)
\right)\,,
\end{equation}
with $\onebold=(1,1,\ldots,1)\in \RR^d$. The quantity $n(z)$ occurs in
the error terms in the final estimates in Theorems \ref{thm-Qa}
and~\ref{thm-QEa}, describing the dependence on~$z\in \uhp^d$.

\begin{lem}\label{lem-ape}For all sufficiently small
$\dt_1>0,\ldots, \dt_d>0$, we have for $0\in \RR^d$ and
$\dt=(\dt_j)_j$:
\[ \cnt(z;0,\dt) \; \ll_\Gm\; n\bigl(\dt^{-1/2},z\bigr)\,,\]
where $\dt^{-1/2}=(\dt_j^{-1/2})_j$.
\end{lem}
\begin{proof}
It suffices to consider $z$ in a fundamental domain $\fd$ chosen as
indicated above. As long as $z$ stays in a compact region, the value
of $\cnt(z;0,\dt)$ is at most the maximal order of totally elliptic
elements of $\Gm$ provided we take the $\dt_j>0$ sufficiently small.
In the previous proof we have seen that this maximal order is bounded
for each~$\Gm$. This proves the lemma for cocompact~$\Gm$.

For other $\Gm$ we fix $B>2A$. If $z\in C_B$, with $C_B$ as in
\eqref{fddc}, we have $\cnt(z;0,\dt)=\oh_\Gm(1)$ for all sufficiently
small $\dt$. Suppose now that $z\in g_\k \siegel(X_\k,B,V_\k)$. If
the $\dt_j$ are sufficiently small, then all $\gm\in \Gm$ such that
$u\bigl((\gm z)_j,z_j\bigr) \leq \dt_j$ lie in $\Gm\cap P_\k$, where
$P_\k$ is the parabolic subgroup fixing~$\k$.

For the remaining computations, we can assume that $\k=\infty$ and
$g_\k=1$. Denote by $N=\bigl\{ \matc1x01\in G \bigr\}$ the unipotent
radical of~$P_\infty = \left\{ \matc tx0{1/t}\in G\right\}$. Elements
in $N_\Gm=N\cap\Gm \subset \Gm_\infty = P_\infty\cap\Gm$ have the
form $\matc 1\om01$ with $\om$ running through a lattice
$\Ld \subset\RR^d$. The quotient $\Gm_\infty/N_\Gm$ is represented by
elements of the form $\matc \e{\al/\e}0{1/\e}$, where $\al\bmod\Ld$
is determined by~$\e$, and where $\e$ runs through a discrete
subgroup of $(\RR^\ast)^d$ such that $\e^2 \Ld =\Ld$, and such that
the $\bigl(\log|\e_1|,\ldots,\log|\e_d|\bigr)$ run through a lattice
in the hyperplane $\sum_j x_j=0$ in~$\RR^d$.

We need a bound for the number of $\gm\in \Gm_\infty$ with
\[ u\bigl( \e_j^2(z_j+\al_j),z_j\bigr)\= \frac{(\e_j^2-1)^2
y_j^2+\bigl(
(\e_j^2-1)x_j+\al_j\bigr)^2}{4\e_j^2y_j^2} \; \leq \; \dt_j\,,\]
for $\dt_j\in (0,1)$ for all $j$. Hence the following quantities have
to be non-negative:
\begin{equation}\label{rhs}4\dt_j
\e_j^2y_j^2-(\e_j^2-1)^2y_j^2\,.\end{equation}
This implies
\[ \log\bigl( 1+2\dt_j - 2 \sqrt{\dt_j+\dt_j^2} \bigr) \;\leq 2\;
\log|\e_j| \;\leq\; \log\bigl( 1+2\dt_j + 2 \sqrt{\dt_j+\dt_j^2}
\bigr)\,.\]
Since $\log|\e|$ runs through a lattice in a hyperplane in $\RR^d$,
this leaves $\oh(1)$ possibilities for the choice of~$\e$. Taking the
maximum of the quantity in \eqref{rhs}, we find for all~$j$:
\[ \bigl| (\e_j^2-1) x_j + \al_j \bigr| \; \leq \;
2\sqrt{\dt_j+\dt_j^2}\; y_j\,.\]
Since $\al$ runs through a coset modulo the lattice~$\Ld$, this gives
at most
\[\oh\biggl( \prod_j (1+\sqrt{\dt_j}\; y_j)\biggr)\ll \prod_j
n_j(\dt_j^{-1/2},z)\]
 possibilities for the choice of $\om$.

For $z\in g_\k \siegel(X_\k,B,V_\k)$ replace $y_j$ by
$\im g_{\k,j}^{-1} z_j$. Together with the bound $\oh(1)$ for
$z\in C_B$, we get the statement in the lemma.
\end{proof}

\section{The Selberg transform and spectral
estimates}\label{sect-st-se}

If $k_1,\ldots,k_d$ are bounded functions on $[0,\infty)$ with compact
support, then the sum
\begin{equation}
K(z,w) \= \sum_{\gm\in\Gm} \prod_j k_j\bigl( u((\gm z)_j,w_j\bigr)
\end{equation}
converges absolutely, and defines a function on
$(\Gm\backslash\uhp^d) \times(\Gm\backslash\uhp^d) $. If we take each
$k_j$ equal to the characteristic function of the interval
$[U_j,V_j)$, then
\[ K(z,z) \= \cnt(U,V;z)\,.\]

It will turn out preferable to use smooth $k_j$, so we will take for
the $k_j$ approximations of those characteristic functions. In this
case $K(z,z)$ is only an approximation of $\cnt(U,V;z)$, but its
spectral expansion as an element of $L^2(\Gm\backslash\uhp^d)$
converges pointwise, and we can write
\[ K(z,z) \= K_\expl(z,z)+ K'(z,z)\]
for each $z\in \uhp^d$, where $K_\expl(z,z)$ is the contribution to
the spectral expansion of a finite number of $\ps_\ell$ (among them
$\ps_0$), and where $K'(z,z)$ is the remainder. The main idea is that
$K_\expl(z,z)$ will yield the explicit terms in the asymptotic
expansion of $\cnt(U,V;z)$, and that estimates of the difference
$K(z,z) - \cnt(U,V;z)$ and of $K'(z,z)$ will contribute to the error
term.

To carry this out, we have to see how the spectral expansion depends
on the functions~$k_j$. That leads us to a study of the Selberg
transform (\S\ref{sect-St}). We also have to know what is the size of the
contributions of various parts of the spectrum
(\S\ref{sect-sm}). In this section we state the results that we need,
and refer for most proofs to \S\ref{sect-eSt} and~\S\ref{sect-spf}.

\subsection{The Selberg transform}\label{sect-St}
We can do most of the work on $\uhp^d$ factor by factor. So we work
first on~$\uhp$.

Functions $k$ on $[0,\infty)$ yield kernel operators on functions $f$
on~$\uhp$:
\begin{equation}\label{convdef}
L_k f(z) \= \int_\uhp k(u(z,w))\, f(w)\, d\mu(w)\,,
\end{equation}
where $d\mu(w) = \frac{d\re w\; d\im w}{(\im w)^2}$ is the invariant
measure associated to the Riemannian metric on~$\uhp$. We take
$k\in C_c^\infty[0,\infty)$. (This implies in particular that all
derivatives are well defined and continuous at $u=0$.) We assume that
the function $f$ is continuous. That suffices for the convergence
in~\eqref{convdef}.

The \emph{Selberg transform} associates to the function $k\in
C_c^\infty[0,\infty)$ an even holomorphic function $h$ on~$\CC$,
given by the following three steps:
\begin{alignat}2\label{Se}
q(p) &\= \int_p^\infty k(u)\, \frac{du}{\sqrt{u-p}}\,,& \text{ for
}&p\geq 0\,,
\displaybreak[0]
\\
\nonumber
g(r) &\= 2q\left( (\sinh(r/2)) ^2 \right)\,,& \text{ for }&r\in \RR\,,
\displaybreak[0]\\
\nonumber
h(\tau) &\= \int_{-\infty}^\infty e^{r\tau} g(r)\, dr\,,& \text{ for
}& \tau\in \CC \,.
\end{alignat}
See, e.g., \cite{Iw95}, p.~33, but note that Iwaniec uses $ir$ as the
variable in~$h$. (See also \cite{Se62}.) The relation can be
described in one step:
\begin{equation}\label{h-alt}
h(\tau ) \= \int_\uhp k(u(z,i)) \, y^{\frac12-\tau}\, d\mu(z)\,,
\end{equation}
which can be made more explicit by use of a hypergeometric function
\begin{equation}
h(\tau) \= 4\pi \int_0^\infty k(u)\,
\hypgeom(\txtfrac12+\tau,\txtfrac12-\tau;1;u)\, du\,.
\end{equation}
See (1.62') and the proof of Theorem~1.16 in~\cite{Iw95}. In fact, $h$
is the \emph{spherical transform} of $k$. See, e.g., \cite{La75},
Chap.~V, \S4. We have in particular
\begin{equation}\label{h1/2}
h\bigl(\txtfrac12\bigr) \= 4\pi\int_0^\infty k(u)\, du\,.
\end{equation}

The Selberg transform has the important property that if
 $\Dt f = \bigl( \frac14-\tau^2\bigr)f$, then
\begin{equation}
L_k f \= h(t)\, f
\end{equation}
(Theorem~1.16 in~\cite{Iw95}).
\medskip

Next we consider $h_1,\ldots,h_d\in C_c^\infty[0,\infty)$, and form
the kernel function
\begin{equation}
\label{kdef}
k(z,w) \= \prod_j k_j(u(z_j,w_j))
\end{equation}
on $\uhp^d\times\uhp^d$. Thus we have the operator
\begin{equation}
L_k f\, (z) \= \int_{\uhp^d} k(z,w)\, f(w)\, d\mu(w)\,,
\end{equation}
with $d\mu=\prod_j d\mu_j$ the product of the invariant measures. This
converges absolutely if $f$ is continuous on~$\uhp^d$. If moreover we
have $\Dt_j f = \left(\frac14-\tau_j^2\right)f$ for the local Laplace
operators $\Dt_j = - y_j^2\frac{\partial^2}{\partial x_j^2} -
y_j^2\frac{\partial^2}{\partial y_j^2}$, then
\begin{equation}\label{khDt}
\Dt_j\bigl( L_k f\bigr)\= h_j(\tau_j) f\qquad\text{ for each }j\,,
\end{equation}
where $h_j$ is the Selberg transform of~$h_j$.

By Lemma~\ref{lem-apl} the sum
\begin{equation}
\label{Kdef}
K(z,w) \;\isdef\; \sum_{\gm\in \Gm} k(\gm z,w)
\end{equation}
converges absolutely, and defines a function in
$C^\infty\bigl((\Gm\backslash\uhp^d)\times(\Gm\backslash\uhp^d)\bigr)$
that satisfies
\begin{equation}\label{K-est}
K(z,w) \= \oh_z(1)\,.
\end{equation}
The boundedness of $K(z,w)$ is uniform for $z$ varying in compact
sets. If $f$ is square integrable on $\Gm\backslash\uhp^d$ for the
invariant measure $d\mu$ then
\begin{equation}
\Kcal_k f \; (z) \= \int_{\Gm\backslash \uhp^d} K(z,w)\, f(w)\,
d\mu(w)
\end{equation}
converges absolutely, and defines an operator
\[\Kcal_k:L^2(\Gm\backslash\uhp^d)
\longrightarrow C^\infty(\Gm\backslash\uhp^d)\,,\]
where $f\mapsto \Kcal_kf(z) $ is continuous on
$L^2(\Gm\backslash\uhp^d)$ for each $z\in \uhp^d$.

\subsection{Spectral decomposition}\label{sect-spd}
A consequence of the irreducibility assumption for the lattice~$\Gm$
is that the spectral theory $L^2(\Gm\backslash\uhp^d)$ is well known.
The Hilbert space
$L^2(\Gm\backslash\uhp^d) = L^2(\Gm\backslash\uhp^d,d\mu)$ has a
spectral decomposition in terms of automorphic forms. In the
cocompact case, each element can be written in $L^2$-sense as
\begin{equation}\label{spd-d} \sum_{\ell\geq 0} a_\ell
\,\ps_\ell\,,\end{equation}
where the $\ps_\ell$ form a complete orthonormal system in
$L^2(\Gm\backslash\uhp^d)$ of simultaneous eigenfunctions of the
$\Dt_j$:
\begin{equation} \Dt_j \ps_\ell \= \bigl(\txtfrac14 -
\tau_{\ell,j}^2\bigr)
\ps_\ell\,,
\end{equation}
with $ \tau_{\ell,j} \in i[0,\infty) \cup (0,\frac12]$. Among these
eigenfunctions we choose
$\ps_0=\frac1{\sqrt{\vol(\Gm\backslash\uhp^d)}}$, a constant
function; hence $\tau_{0,j}=\frac12$ for all~$j$. For each
$\ell\geq 1$ we know that
$\tau_{\ell,j}\in i[0,\infty) \cup (0,\frac12)$. The $a_\ell$ form a
sequence in the Hilbert space~$\ell^2$.

If $\Gm$ has cusps, there is a subspace
$L^{2,\mathrm{discr}}(\Gm\backslash\uhp^d)$ with the same structure
as in the cocompact case. It always contains the constant function
$\ps_0$. If $d=1$ there may be finitely many $\ell\geq 1$ for which
$\ps_\ell$ is a residue of an Eisenstein series and at most
countably many $\ps_\ell$ that are cusp forms. The orthogonal
complement $L^{2,\mathrm{cont}}(\Gm\backslash \uhp^2)$ is a sum of
direct integrals. Elements of this space can be written in
$L^2$-sense in the form
\begin{equation}\label{spd-c}
\sum_\k 2c_\k \sum_{\mu \in \Lcal_\k} \int_0^\infty b_{\k,\mu}(t)\,
E(\k;it,i\mu)\, dt\,.
\end{equation}
Here $\k$ runs over representatives of the finitely many cuspidal
$\Gm$-orbits, the $c_\k$ are positive constants, $\Lcal_\k$ is a
lattice in the hyperplane $\sum_j x_j=0$ in $\RR^d$, and
$E(\k;s,i\mu)$ is an Eisenstein series, satisfying
\[\Dt_j E(\k;s,i\mu)
= \bigl( \txtfrac14
-(s+i\mu_j)^2\bigr) E(\k;s,i\mu)\]
for each~$j$. For $f\in L^2(\Gm\backslash\uhp^d)$ we have
$a_\ell = (f,\ps_\ell)$. If $f$ is bounded and sufficiently smooth,
then $b_{\mu,\k}$ is given by integration against
$\overline{E(\k;it,i\mu)}$.

The quantity $\hat\tau(\Gm) = \sup_{\ell\geq 1,\, 1\leq j \leq d}\re
\tau_{\ell,j}$ in \eqref{tau-hat} is related to the quantity
$p(\Gm\backslash G)\in [2,\infty)$ in \cite{KeSa}, with
$G=\PSL_2(\RR)^d$, by
\begin{equation}\label{p-th-comp}p(\Gm\backslash G) \;\geq\;
\frac1{\frac12-\hat\tau}\quad\text{ or equivalently }\quad \tau\;\leq
\; \frac12-\frac1{p(\Gm\backslash G)}\,.
\end{equation}
So $\hat\tau=\frac12$ would imply $p(\Gm\backslash G)=\infty$ (no
strong spectral gap), and $p(\Gm\backslash G)=2$ implies $\hat\tau=0$
(no exceptional eigenvalues at all). We have to be careful to use
inequalities in~\eqref{p-th-comp}. Kelmer and Sarnak take all
irreducible representations of $G=\PSL_2(\RR)^d$ in
$L^{2,\mathrm{discr}}(\Gm\backslash G)$ into account. Such a
representation is visible in
$L^{2,\mathrm{discr}}(\Gm\backslash \uhp^d)$ only if all
$d$~components of the representation have a non-trivial
$\PSO(2)$-invariant vector. We recall that in the congruence case
(including all non-cocompact $\Gm$ if $d\geq 2$) we have
$\hat\tau(\Gm)\leq \frac19$. For all cocompact $\Gm$ we have
$\hat\tau(\Gm)<\frac12$.
\medskip

We return to the kernel function $K$ in~\eqref{Kdef}. By \eqref{khDt}
and the invariance of the kernel $k(z,w)$, we have for fixed
$z\in\uhp^d$:
\[ \int_{\Gm\backslash \uhp^d} K(z,w)\, \overline{\ps_\ell(w)}\,
d\mu(w)
\= \int_{\uhp^d} k(z,w)\, \overline{\ps_\ell(w)} \, d\mu(w) \=
h(\tau_\ell)
\overline{\ps_\ell(z)}\,, \]
with
\begin{equation}\label{hdef}
h(\tau) \= \prod_j h_j(\tau_j)\,.
\end{equation}
Therefore the scalar product of $K(z,\cdot)$ with $\ps_\ell$ makes
sense. If $\Gm$ is not cocompact, we find in a similar way that the
coefficients $b_{\k,\mu}(t)$ in \eqref{spd-c} are given by
$\prod_j h_j(it+\nobreak i\mu_j)\; \overline{E(\k;it,i\mu)}$. Thus we
obtain the spectral expansion of
$K(z,\cdot) \in L^2(\Gm\backslash\uhp^d)$:
\begin{align}\label{Kdc}
K(z,\cdot) &\= \sum_{\ell} h(\tau_\ell)\, \overline{\ps_\ell(z)}\,
\ps_\ell\\
\nonumber
&\qquad\hbox{}
 + \sum_\k 2c_\k \sum_{\mu\in \Lcal_\k} \int_0^\infty h(it+i\mu) \,
\overline{E(\k;it,i\mu;z)}\, E(\k;it,i\mu)\, dt\,.\end{align}
In the cocompact case, we understand the sum over $\k$ to be absent.

This spectral expansion converges in the Hilbert space
$L^2(\Gm\backslash\uhp^d)$. To use it to investigate the counting
function $\cnt(U,V;z)$ in the way indicated in the introduction of
this section, we need it to make sense pointwise.

\begin{thm}\label{thm-pspe}Let $f\in  C^{2d}(\Gm\backslash \uhp^d)$ be
bounded, and suppose that the derivatives
$\Dt_1^{a_1}\Dt_2^{a_2}\cdots \Dt_d^{a_d} f$ are bounded for all
 choices of $a_j\in \{0,1,2\}$. Then the spectral expansion of $f$
converges absolutely and uniformly on compacta.

In particular, if the $k_j$ in~\eqref{kdef} are in
$C_c^\infty[0,\infty)$ for all~$j$, then the expansion
\begin{align}\label{K-spe}
K(z,w) &\= \sum_\ell h(\tau_\ell)\, \overline{\ps_\ell(z)} \,
\ps_\ell(w)\\
\nonumber
&\qquad\hbox{}
+ \sum_\k 2c_\k \sum_{\mu\in \Lcal_\k} \int_0^\infty h(it+i\mu)\,
\overline{ E(\k;it,i\mu;z)}\, E(\k;it,i\mu;w)\, dt
\end{align}
converges absolutely for each choice $z,w\in \uhp^d$.
\end{thm}
This result is more or less standard. We will sketch a proof
in~\S\ref{sect-spe}.

\subsection{Spectral measure}\label{sect-sm}
As indicated in the introduction of this section, we will need to know
how the various parts of the spectral set
\[ \biggl( i\RR \cup\bigl(0,\frac12\bigr]\biggr)^d\]
contribute to the spectral expansion of $K(z,z)$. We write $i\RR$
instead of $i[0,\infty)$, since in the term in the spectral
expansion~\eqref{K-spe} corresponding to the continuous spectrum
there are quantities $i\left( t+\mu_j\right)$, which in some cases
are in $i(-\infty,0)$.

For $X\in [1,\infty)^d$ we put
\begin{equation}
\label{Ydef}
Y(X)\= \prod_j \biggl( \bigl(0,\frac12\bigr]\cup i(-X_j,X_j)\biggr)\,,
\end{equation}
and define
\begin{equation}
\label{SX-def}
\begin{aligned}
S(X;z,w) &\= \sum_{\ell\,,\; t_\ell \in Y(X)} \overline{\ps_\ell(z)
} \ps_\ell(w) \\
\nonumber
&\qquad\hbox{}
+ \sum_\k 2c_\k \sum_{\mu\in \Lcal_\k} \int_{t\geq 0\,,\; (t+\mu_j)_j
\in Y(X)} \overline{E(\k;it,i\mu;z)}\, E(\k;it,i\mu;w)\, dt\,.
\end{aligned}
\end{equation}
This is a smooth function on
$(\Gm\backslash\uhp^d)\times(\Gm\backslash\uhp^d)$. (The sum over
$\ell$ is finite. The region of integration is finite for all
$(\k,\mu)$, and empty for almost all $(\k,\mu)$.)

The following estimate of the spectral function $S(X;z,z)$ will play
an important role in \S\ref{sect-pfs} in the proof of our main
results:
\begin{thm}\label{thm-spm}For $X\in [1,\infty)^d$ and $z\in \uhp^d$:
\[ S(X;z,z) \ll_{\Gm} X_1^2 X_d^2\cdots X_d^2\, n(X,z)\,.\]
\end{thm}
The quantity $n(X,z)$ has been defined in~\eqref{ndef}. It makes
explicit the dependence of the spectral measure on the
point~$z\in\uhp$. This constitutes a difference with~\cite{BMW99},
where we used a result of H\"ormander to estimate the spectral
measure uniformly for $z$ in compact sets, obtaining an asymptotic
formula for the lattice point counting function on symmetric spaces
of rank one that was uniform for~$z$ varying in compact sets only.

We prove Theorem~\ref{thm-spm} in~\S\ref{sect-spm}. The proof uses the
a priori estimate of the lattice point counting function in
Lemma~\ref{lem-ape}.

\section{Proof of the lattice points theorems}\label{sect-pfs}
This section is the heart of this paper, where we carry out the plan
sketched in the introduction of~\S\ref{sect-st-se}. We consider the
asymptotic behavior of the quantity $\cnt(U,V;z)$ defined
in~\eqref{cntdef}, with $U,V\in [0,\infty)^d$, $U_j<V_j$ for all~$j$.
In Theorem~\ref{thm-Qa}, which is slightly more general than
Theorem~\ref{thm-Qai} in the introduction, all $V_j$ tend
to~$\infty$, whereas in Theorem~\ref{thm-QEa}, generalization of
Theorem~\ref{thm-QEai}, some intervals $[U_j,V_j)$ stay fixed. For
most of the section we handle the proofs simultaneously.

\subsection{The main parameters}\label{sect-mp}
We partition the set
$\{1,\ldots,d\}$ into two disjoint subsets $Q$ and~$E$, with the
requirement that $Q$ is non-empty.

For each $j\in Q$ we let $V_j\geq 1$ tend to~$\infty$, and choose
$U_j=0$ or $U_j=\frac12 V_j$. (This choice may depend on the place
$j\in Q$.)
The simplest is to take all $V_j$ with $j\in Q$ equal to each other.
We wish also to include the case that $V_j=T^{a_j}$ with positive
exponent $a_j$, where $T$ tends to infinity. However, we do not let
the $V_j$ run apart too much, by fixing a parameter $\hat q$
satisfying
\begin{equation}\label{hatq}
\hat q\geq \#Q\,,\quad\text{and require}\quad \min_{j\in Q} V_j ^{\hat
q} \= \prod_{j\in Q} V_j\,.
\end{equation}
Thus, if all $V_j$ with $j\in Q$ are equal to each other, then
$\hat q=\#Q$. For each $j\in E$ we keep the non-empty interval
$[U_j,V_j)$ fixed.

These are the parameters used in the Theorems~\ref{thm-Qa}
and~\ref{thm-QEa}. They constitute the ``main parameters'' in
Table~\ref{tab-parms}.
\begin{table}[ht]
\[ \renewcommand\arraystretch{1.3}
\begin{array}{|cc|l|}\hline
\multicolumn{3}{|c|}{\text{\it Main parameters}}\\ \hline
Q&& \text{a non-empty subset of $\{1,\ldots,d\}$}\\
E&& \text{the complement }\{1,\ldots,d\}\setminus Q \\ \hline
V_j&j\in Q&V_j\geq 2\,,\;V_j\rightarrow\infty\,,\; \\
&j\in E&V_j>0\text{ fixed}\\ \hline
U_j&j\in Q& Q_j=0\text{ (fixed), or }Q_j=\frac12 V_j \rightarrow
\infty\\
&j\in E&U_j\in[0,V_j) \text{ fixed }\\ \hline
\hat q& &\hat q\geq \#Q\text{ fixed}\,,\quad\hat q \log \vmin =
\sum_{j\in Q} \log V_j \\
\vmin&& \vmin=\min_{j\in Q}V_j\geq 2 \\ \hline
\multicolumn{3}{|c|}{\text{\it Auxiliary parameters}}\\ \hline
\th &&\th\in(0,1)\\ \hline
Y_E&&Y_E\in(0,1]\,,\; Y_E\downarrow0\\ \hline
Y_j & j\in Q& Y_j = V_j^\th\,,\; Y_j \geq 1\,,\; Y_j\rightarrow\infty
\\
& j\in E& Y_j=Y_E \downarrow0\\
c&&0<c<\frac12\,, \text{ if }\hat\tau>0 \text{ then } c<\hat\tau \\
\hline
\end{array}
\]
\caption{Overview of the parameters
in~\S\ref{sect-pfs}.}\label{tab-parms}
\end{table}

\subsection{Test functions and auxiliary parameters}\label{sect-ap} In
the introduction of \S\ref{sect-st-se} we have sketched our plan to
prove the main results in~\S\ref{sect-Qa} and~\ref{sect-QEa}. We take
the approximations $h_j\in C_c^\infty(0,\infty)$ of the
characteristic functions $[U_j,V_j)$ in the following way:
\begin{equation}\label{k-prop}
\begin{aligned}
&0\;\leq\; k \;\leq \;1\,,\qquad k_j^{(l)}\= \oh_l(Y^l)\text{ for
}l\in \NN\,,\\
k_j&\=1\text{ on }\begin{cases}
[U_j+Y_j,V_j-Y_j]&\text{ if }U_j>0\,,\\
[0,V_j-Y_j]&\text{ if }U_j=0\,,
\end{cases}\\
k_j&\=0\text{ on } \begin{cases}
[0,U_j-Y_j]\cup[V_j+Y_j,\infty)&\text{ if }U_j>0\,,\\
[V_j+Y_j,\infty)&\text{ if }U_j=0\,.\end{cases}
\end{aligned}
\end{equation}
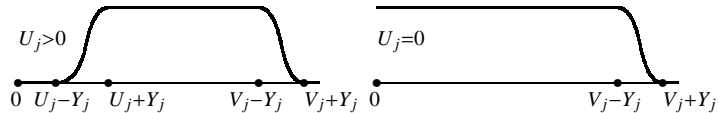
\[
\setlength\unitlength{1cm}
\begin{picture}(4,1.5)(0,.2)
\put(0,1){$\scriptstyle U_j>0$}
\put(0,.5){\line(1,0){4}}
\put(0,.5){\circle*{.1}}
\put(-.1,.2){$\scriptstyle 0$}
\put(.5,.5){\circle*{.1}}
\put(.2,.2){$\scriptstyle U_j-Y_j$}
\put(1.2,.5){\circle*{.1}}
\put(1.2,.2){$\scriptstyle U_j+Y_j$}
\put(3.2,.5){\circle*{.1}}
\put(2.8,.2){$\scriptstyle V_j-Y_j$}
\put(3.8,.5){\circle*{.1}}
\put(3.8,.2){$\scriptstyle V_j+Y_j$}
\thicklines
\put(0,.5){\line(1,0){.5}}
\qbezier(.5,.5)(.8,.5)(.9,1)
\qbezier(.9,1)(1,1.5)(1.2,1.5)
\put(1.2,1.5){\line(1,0){2}}
\qbezier(3.8,.5)(3.6,.5)(3.5,1)
\qbezier(3.5,1)(3.4,1.5)(3.2,1.5)
\put(3.8,.5){\line(1,0){.2}}
\end{picture}
\qquad
\begin{picture}(4,1.5)(0,.2)
\put(0,1){$\scriptstyle U_j=0$}
\put(0,.5){\line(1,0){4}}
\put(0,.5){\circle*{.1}}
\put(-.1,.2){$\scriptstyle 0$}
\put(3.2,.5){\circle*{.1}}
\put(2.8,.2){$\scriptstyle V_j-Y_j$}
\put(3.8,.5){\circle*{.1}}
\put(3.8,.2){$\scriptstyle V_j+Y_j$}
\thicklines
\put(0,1.5){\line(1,0){3.2}}
\qbezier(3.8,.5)(3.6,.5)(3.5,1)
\qbezier(3.5,1)(3.4,1.5)(3.2,1.5)
\put(3.8,.5){\line(1,0){.2}}
\end{picture}
\]
The parameters $Y_j$ control how quickly $k_j$ changes from~$0$ to~$1$
and back. We require
\begin{equation}\label{Ycond}
Y_j \;\leq\; \frac{V_j-U_j}2\quad\text{ if }U_j\;>\;0\,,\qquad Y_j
\;\leq\; \frac12V_j\quad\text{ if }U_j\;=\;0\,.
\end{equation}
If one fixes a smooth function $\om \in C^\infty(\RR)$ that increases
from $0$ to~$1$ on an interval contained in $(0,1)$, then
$h_j(u) = \om\bigl(
(u-U_j)/Y_j\bigr)$ on $[U_j,U_j+\nobreak Y_j]$ satisfies on this
interval the condition on the derivatives, and goes from $0$ to~$1$.
On $[V_j-\nobreak Y_j,V_j]$ we proceed similarly. This gives a choice
such that $k_j=0$ outside~$[U_j,V_j]$. We can equally well arrange
that $k_j=1$ on $[U_j,V_j]$.

The $Y_j$ are new parameters. They play a role in the proof, not in
the theorems. At the end of the proof we try to choose them
optimally. To avoid having to keep track of too many auxiliary
parameters, we assume from the start that $Y_j = V_j^\th$ for
$j\in Q$, with $\th\in (0,1)$ a single auxiliary parameter. So
the~$Y_j$ are large parameters for $j\in Q$.

We let the $Y_j$ with $j\in E$ tend to zero. It seems that we do not
loose much if we take all these parameters equal to a quantity $Y_E$
tending to~$0$, for which we require that $2Y_E \leq V_j-U_j$ for all
$j\in E$, and $2Y_E\leq U_j$ for all $j\in E$ with~$U_j>0$.

The estimates of Selberg transforms in \S\ref{sect-eSt} depend on a
small positive parameter~$c\in \bigl(0,\frac12\bigr)$. In
Lemma~\ref{lem-se2}~c) there is also a small positive parameter
$\dt$. We choose $\dt$ such that $\dt< U_j$ for all~$j\in E$ with
$U_j>0$. The dependence on $\dt$ of the implicit constants in the
estimates is absorbed in the dependence of the choice of the
intervals $[U_j,V_j]$ with $j\in E$. The positive constant $c$ will
turn up in exponents in some estimates. We take $c<\hat\tau$ if
$\hat\tau>0$. (We recall that $\hat\tau$ measures the spectral gap,
which is maximal if $\hat\tau=0$.)

\subsection{Terms in the asymptotic formula}With the test functions
in~\eqref{k-prop} we form the $\Gm$-invariant kernel
\begin{equation}\label{Kdef1}
K(z,w) \= \sum_{\gm\in \Gm} k\bigl( u(z,w)\bigr)\,,\qquad k(u) \=
\prod_j k_j(u_j)\,,
\end{equation}
as indicated in the introduction of~\S\ref{sect-st-se}. The diagonal
value $K(z,z)$ gives an approximation of $\cnt(U,V;z)$. By
$h(\tau) = \prod_j h_j(\tau_j)$ we denote the product of the Selberg
transforms of the~$k_j$.

{}From the absolutely convergent spectral decomposition
in~\eqref{K-spe} we single out
\begin{equation}
K_\expl(z,z) \;:=\; h(\tau_0) \,\frac1{\vol(\Gm\backslash \uhp^d)} +
\sum_{\ell\geq 1\,,\; \forall_j\; 0<\tau_{\ell,j}<\frac12}
h(\tau_\ell)\, \bigl|\ps_\ell(z)
\bigr|^2\,.
\end{equation}
Here we have made the choice to take not only the contribution of the
constant functions, but also all of the terms corresponding to
totally exceptional eigenvalues. This term $K_\expl(z,z)$ will lead
to the explicit term in our asymptotic expansions.

As the explicit term in the final results we use
\begin{equation}\label{Ecal-def}
\begin{aligned}
\Ecal(U,V;z) &\= \sum_{\ell\geq 0\,,\; \forall_j\; \tau_{\ell,j} \in
(0,\frac12]} \bigl|\ps_\ell(z)\bigr|^2 \,\prod_{j\in E}
\eta(U_j,V_j;\tau_{\ell,j})\\
&\qquad\qquad\hbox{} \cdot
\prod_{j\in Q}\frac{\sqrt\pi \, 2^{1+2\tau_{\ell,j}}
\,\Gf(\tau_{\ell,j})}{\Gf\bigl(\frac32+\tau_{\ell,j}\bigr)}\bigl(
V_j^{\frac12+\tau_{\ell,j}} - U_j^{\frac12+\tau_{\ell,j}}\bigr)
\,,
\end{aligned}
\end{equation}
where
\begin{equation}\label{eta-def}
\eta(a,b;\tau) \= \int_{z\in \uhp\,,\; a\leq u(z,i) <b}
y^{\frac12+\tau}\, d\mu(z)
\end{equation}
is the Selberg transform of the characteristic function of~$[a,b)$. So
we will need to estimate the difference between $\Ecal(U,V;z)$ and
$K_\expl(z,z)$. We note that there might not exist totally
exceptional eigenvalues for the group~$\Gm$. In that case
$\Ecal(U,V;z)$ is equal to the term for~$\ell=0$:
\begin{equation}\label{Ecal-const}
\frac1{\vol(\Gm\backslash\uhp^d)}\; (4\pi)^d \prod_j \bigl(
V_j-U_j\bigr) \,.
\end{equation}
(See \eqref{eta-b} in Lemma~\ref{lem-eta-est}.)

The remaining part of the spectral decomposition is split up according
to subsets $Z(n)$ of the space of spectral parameters. For
$n\in \NN^d$ we put
\begin{equation}
\label{Zndef}
\begin{aligned}
Z(n) &\= \biggl\{ \tau \in \bigl( i[0,\infty)\cup(0,\frac12)
\bigr)^d\;:\;
\\
&\qquad \qquad \tau_j \in i[-n_j,1-n_j]\cup i[n_j-1,n_j)\text{ if }n_j
>1\,,\\
&\qquad \qquad \tau_j \in (0,\frac12 )\cup i(-1,1)\text{ if
}n_j=1\,\biggr\}\,.
\end{aligned}
\end{equation}

For $n\in \NN^d$, $n\neq \onebold = (1,1,\ldots,1)$ we define
\begin{equation}
\label{Kndef}
\begin{aligned}
K_n(z,z) &\= \sum_{\ell\geq 1\,,\; \tau_\ell \in Z(n) } h(\tau_\ell)\,
\bigl|\ps_\ell(z)\bigr|^2\\
&\qquad\hbox{}
+ \sum_\k 2c_\k \sum_{\mu\in \Lcal_\k} \int_{t\geq 0\,,\; i(t+\mu) \in
Z(n)} h(it+i\mu)\, \bigl| E(\k;it,i\mu;z)\bigr|^2\, dt\,.
\end{aligned}
\end{equation}
For $n=\onebold$ we modify the term from the discrete spectrum by
requiring not only $\tau_\ell \in Z(\onebold)$, but also
$\tau_{\ell,j} \in i[0,\infty)$ for some~$j$. (The totally
exceptional terms go into $\Ecal(U,V;z)$.) With these definitions, we
have
\begin{equation}
K(z,z) - K_\expl(z,z) \= \sum_{n\in \NN^d} K_n(z,z)\,.
\end{equation}
It will be hard work to estimate this sum.

Finally, we also will have to estimate the difference between $K(z,z)$
and the counting quantity $\cnt(U,V;z)$. Table~\ref{tab-est} gives an
overview of the estimates to be carried out.
\begin{table}\renewcommand\arraystretch{1.3}
\[
\begin{array}{|rcll|} \hline
\multicolumn{4}{|l|}{\cnt(U,V;z) \;=\; \Ecal(U,V;z) + \oh
\bigl(\mathit{Err}_1 + \mathit{Err}_2+\mathit{Err}_3 \bigr)}\\ \hline
\mathit{Err}_1 &\text{estimate of}& \sum_{n\in \NN^d} K_n(z,z) &\text{
in \S\ref{sect-sumsp}}\\
\mathit{Err}_2 &\text{estimate of}& K_\expl(z,z) - \Ecal(U,V;z)
&\text{ in \S\ref{sect-texpl}}\\
\mathit{Err}_3 &\text{estimate of}& \cnt(U,V;z)-K(z,z) &\text{ in
\S\ref{sect-shsm}}\\ \hline
\end{array}\]
\caption{Overview of the error term estimates.} \label{tab-est}
\end{table}

\subsection{The explicit term}\label{sect-texpl}
The explicit term $\Ecal(U,V;z)$ in~\eqref{Ecal-def} is a finite sum.
Each of its terms contains as a factor, for $j\in E$, the Selberg
transform $\eta(U_j,V_j;\tau_{\ell,j})$ of the characteristic
function of $[U_j,V_j)$, and, for $j\in Q$, the approximation of this
Selberg transform given in~\eqref{eta-as} in Lemma~\ref{lem-eta-est}.
That approximation is uniform on intervals $[c,\frac12]$ for each
$c>0$. Here we want to apply it with $\tau$ equal to the coordinates
$\tau_{\ell,j}$ of the totally exceptional eigenvalues. These
coordinates form a finite subset of $(0,\frac12)$. We take
$c\in (0,\frac12)$ smaller than the minimum of these finitely many
$\tau_{\ell,j}$, and then apply \eqref{eta-as} uniformly. Thus, this
parameter~$c$ depends on the group~$\Gm$, and will lead to an
implicit dependence of the error terms on~$\Gm$.

\begin{lem}\label{lem-KE}The explicit term satisfies
\begin{equation}
\Ecal(U,V;z) - K_\expl(z,z) \;\ll_{\Gm,E} n(z)\, \bigl(
\vmin^{\th-1}+Y_E\bigr)
\,\prod_{j\in Q} V_j\,.
\end{equation}
The factor $n(z)$ has been defined in~\eqref{ndef}. See
Table~\ref{tab-parms} for $Y_E$, $\vmin$ and~$\th$. The $E$ in
$\ll_{\Gm, E}$ indicates an implicit dependence on all $U_j$ and
$V_j$ with $j\in E$.
\end{lem}
\begin{proof}For each of the finitely many $\tau=\tau_\ell$ occurring
in $\Ecal(U,V;z)$ and $K_\expl(z,z)$, we have by
Lemma~\ref{lem-se2}~a):
\begin{equation}\label{diff-eta-h} \prod_j
\eta(U_j,V_j;\tau_j)-\prod_j h_j(\tau_j) \;\ll\; \sum_j \bigl(
\eta(U_j,V_j;\tau_j)-h_j(\tau_j)\bigr)\, \sum_{l\neq j}
\eta(U_j,V_j;\tau_j)\,.\end{equation}
We apply Lemmas~\ref{lem-eta-est} and~\ref{lem-se1} with $c\in
\bigl(0,\frac12\bigr)$ chosen so that $c<\tau_{\ell,j}$ for all~$j$
for all $\ell$ occurring in the explicit term. We find, uniformly for
$c\leq \tau_j\leq \frac12$:
\begin{alignat*}3
&j\in Q:\;& \eta(U_j,V_j;\tau_j) - h_j(\tau_j)&\;\ll_c\; Y_j
V_j^{\tau_j-\frac12} \= V_j ^{\tau_j+\th-\frac12}& \quad&\text{
Lemma~\ref{lem-se1} c)}\,,\\
&& \eta(U_j,V_j;\tau_j) &\;\ll_c\;
V_j^{\tau_j+\frac12}&&\eqref{eta-as}\,,\\
&j\in E:\;&\eta(U_j,V_j;\tau_j) - h_j(\tau_j)&\;\ll_E\; Y_j \=
Y_E&&\text{ Lemma~\ref{lem-se1} c)}\,,\\
&&\eta(U_j,V_j;\tau_j)&\;\ll_E\; 1\,.
\end{alignat*}
Note that we leave implicit the influence of the fixed quantities
$U_j$ and $V_j$ with $j\in E$, but keep their difference~$Y_j$
explicit. The difference in~\eqref{diff-eta-h} is estimated by the
following quantity, uniformly in the $\tau=\tau_\ell$ under
consideration:
\begin{align*}
&\;\ll_{\Gm,E} \sum_{j\in Q} V_j^{\tau_j+\th-\frac12}\cdot \prod_{l\in
Q\setminus\{j\}} V_l^{\tau_l+\frac12}\cdot \oh(1)
+ \sum_{j\in E} Y_E\cdot \prod_{l\in Q} V_l^{\tau_l+\frac12}\cdot
\oh(1)
\displaybreak[0]\\
&\;\ll\; \biggl( \sum_{j\in Q} V_j^{\th-1}+Y_E \biggr) \prod_{l\in Q}
V_l^{\tau_l+\frac12} \;\ll\; \bigl(\vmin^{\th-1}+Y_E\bigr)
\prod_{j\in Q}V_j\,,
\end{align*}
where we have used $\tau_j\leq \frac12$ in the last step.

We still have to estimate the finitely many $\ps_\ell(z)$. If
$\ps_\ell$ is a cusp form or if $\ell=0$, then
$|\ps_\ell(z)| = \oh_\Gm(1)$. If $\ps_\ell$, with $\ell\geq 1$,
arises from a residue of an Eisenstein series, it satisfies
$\ps_\ell(g_\k z) = \oh(N(y)^{\frac12-\r_\ell})$ as
$N(y) = \prod_j y_j \rightarrow\infty$ for all cusps $\k$, for some
$\r_\ell\in (0,\frac12)$. Hence $|\ps_\ell(z)| \ll n(z)^{1/2}$. Since
the explicit term and $K_e(z,z)$ run over finitely many~$\ell$, this
estimate can be used uniformly, thus giving the lemma. Note that here
arises another implicit dependence of the error terms on the
group~$\Gm$.
\end{proof}

\subsection{Sum over the spectrum}\label{sect-sumsp}
We turn to the estimation of $\sum_{n\in \NN^d} K_n(z,z)$, as defined
in~\eqref{Kndef}, with the given modification for $n=\onebold$. We
will use that
\begin{equation} K_n(z,z) \;\leq \; M(n)\, S_{\!n}\,,
\end{equation}
where, with $Z(n)$ as defined in~\eqref{Zndef}
\begin{align}
M(n) &\= \sup_{\tau\in Z(n)} |h(\tau)|\,,\\
\label{Sndef}
S_{\!n} &\= \sum_{\ell\geq 1\,,\; \tau_\ell \in Z(n) }
\bigl|\ps_\ell(z)\bigr|^2\\
\nonumber
&\qquad\hbox{}
+ \sum_\k 2c_\k \sum_{\mu\in \Lcal_\k} \int_{t\geq 0\,,\;
\bigl(i(t+\mu_j)
\bigr)_j \in Z(n)} \bigl| E(\k;it,i\mu;z)\bigr|^2\, dt\,.
\end{align}

\begin{lem}\label{lem-Mest}For each sufficiently small
$c\in \bigl( 0,\frac12\bigr)$ the quantity $M(n)$ has for each
$n\in \NN^d$ and each $l\in \NN^d$ an estimate
\[ M(n) \;\ll_{E,l,c} \prod_j f_{l_j,j}(n_j)\,, \]
where for all $l\in \NN$
\begin{equation}
\begin{aligned}
f_{l,j}(1) &\= \begin{cases}
V_j^{\hat\tau+\frac12}&\text{ if }j\in Q\text{ and }\hat\tau>0\,,\\
V_j^{c+\frac12}&\text{ if }j\in Q\text{ and }\hat\tau=0\,,\\
1&\text{ if }j\in E\,,
\end{cases}\\
f_{l,j}(n)&\=
\begin{cases}
n^{-l-\frac12} Y^{1-l}_j V_j^{l-\frac12}&\text{ if }j\in Q\text{ and
}n\geq 2\,,\\
n^{-l-\frac12} Y^{1-l}_j &\text{ if }j\in E \text{ and }n\geq 2\,.
\end{cases}
\end{aligned}
\end{equation}

If $E=\emptyset$ and $\hat\tau>0$, we have the slightly better
estimate
\begin{equation}
\label{M1e}
M(\onebold) \;\ll_c \; \vmin^{c-\hat\tau} \prod_j f_{l_j,j}(1)\,.
\end{equation}
\end{lem}
We recall that $\hat\tau\in\bigl[0,\frac12\bigr]$ is the supremum of
the real parts $\re \tau_{\ell,j}$, $\ell\geq 1$, of the spectral
parameters. It measures the spectral gap.
\begin{proof}If $\hat\tau >0$, we take $c\in (0,\hat\tau)$. We use the
estimates in~\S\ref{sect-eSt} of the Selberg transforms~$h_j$ of
the~$k_j$ with this value.

If $n=1$ we have to consider $\tau_j \in (0,\hat\tau)\cup i[-1,1]$.
For $\tau\in \bigl[c,\frac12\bigr)$, which can occur only if
$\hat\tau>0$, we use Lemma~\ref{lem-se1}~c) and \eqref{eta-as} in
Lemma~\ref{lem-eta-est} to get
\[ h_j(\tau_j) \;\ll_c\; V_j^{\tau_j+\frac12}
 + Y_j \,m\,,\]
with
\[m\=\max(V_j^{\tau_j-\frac12},U_j^{-\frac12}) \text{ if
}U_j>0\,,\qquad m\=\max(V_j^{-\frac12},V_j^{\tau_j-\frac12})
 \text{ if }U_j\=0\,.\]
If $j\in Q$ we get a bound by $\oh(V_j^{\tau_j+\frac12})=
\oh(V_j^{\hat\tau+\frac12})$. For $j\in E$ the dependence on $U_j$
and $V_j$ is left implicit, so we can use the bound~$1$. If
$|\tau_j|\leq c$, Lemma~\ref{lem-se2}~b) gives the bound
$\oh_c(V_j^{c+\frac12})$ for $j\in Q$, and $\oh_{c,E} $ if $j\in E$.
For $\tau_j\in i\RR$, $c\leq |\tau_j|\leq 1$, we use
Lemma~\ref{lem-se2}~c) with $l=1$. For $j\in E$ we take care to
choose the $\dt$ in Lemma~\ref{lem-se2} such that $U_j\geq\dt$ if
$U_j>0$. We use that $k'\ll Y^{-1}$ to find $\oh(V_j^{\frac12})$ if
$j\in Q$ and $\oh(1)$ if $j\in E$.

If $n\geq 2$ we use Lemma~\ref{lem-se2}~c) and the condition
$k_j^{(l)} = \oh(Y^{-l})$ to obtain the bounds by $f_{l,j}(n)$.

In the case of $M(\onebold)$ we have the additional information that
$\tau_j \in i[-1,1]$ for at least one~$j$. If $E=\emptyset$ this
leads to the estimate in~\eqref{M1e}.
\end{proof}

The problem with $S_{\!n}$ in~\eqref{Sndef} is that we do not have a
direct estimate for it. All we have is Theorem~\ref{thm-spm}, which
gives
\begin{equation}
\sum_{m\in \NN^d\,,\; \forall_j\, m_j\leq n_j } S_{\!m} \;\ll_\Gm\;
n(n,z)
\prod_j n_j^2 \= \prod_j \max(n_j^2,\y_j(z) \,n_j)
\,.
\end{equation}
(See \eqref{nTdef} for $n(n,z)$.)
So we need to carry out a $d$-dimensional partial summation.
\begin{lem}\label{lem-sKne}
Let $z\in \uhp$. For $c\in \bigl(0,\frac12\bigr)$ as in the previous
lemma, we have if $\hat \tau>0$
\begin{equation}
\sum_{n\in \NN^d} K_n(z,z)
\;\ll_{\Gm,E,c}\; n(z) \cdot
\begin{cases}
Y_E^{-\frac12\#E} \prod_{j\in Q} V_j^{1-\frac12\th}
&\text{ if }\th\leq 1-2\hat\tau\,,\\
Y_E^{-\frac12\#E} \prod_{j\in Q} V_j^{\hat\tau+\frac12}
&\text{ if }\th \geq 1-2\hat\tau\\
&\qquad\text{ and }E\neq\emptyset\,,\\
\vmin^{\frac12-\hat\tau-\frac\th2} \prod_{j=1}^d V_j^{\hat
\tau+\frac12}
&\text{ if } 1-2\hat\tau\leq \th\leq 1-2c \\
&\qquad\text{ and }E=\emptyset\,,\\
\vmin^{c-\hat\tau} \prod_{j=1}^d V_j^{\hat \tau+\frac12}
&\text{ if }\th\geq 1-2c \\
&\qquad\text{ and }E=\emptyset\,,
\end{cases}
\end{equation}
and if $\hat\tau=0$
\begin{equation}
\sum_{n\in \NN^d} K_n(z,z)
\;\ll_{\Gm,E,c}\; n(z) \cdot
\begin{cases}
Y_E^{-\frac12 \#E} \prod_{j\in Q} V_j^{1-\frac12\th}
&\text{ if }\th \leq 1-2c\,,\\
Y_E^{-\frac12\,\#E}\, \prod_{j\in Q} V_j^{c+\frac12}
&\text{ if }\th \geq 1-2c\,.
\end{cases}
\end{equation}
\end{lem}
\begin{proof}
We have
$S_{\!n} \= \sum_{H\subset\{1,\ldots,d\}} (-1)^{\#H} S(Y(n-\onebold_H);z,z)$,
where $\onebold_H \in \NN_{\geq 0}^d$ has coordinate $1$ if $j\in H$
and coordinate $0$ otherwise. We understand that $S(Y(m);z,z)$ is
zero if one of the coordinates of $m$ vanishes.

To estimate $\sum_n K_n(z,z)$ it suffices to consider
\begin{equation}\label{sumSM}
 \sum_{m\in \NN^d} S\bigl( Y(m);z,z) \sum_{H \subset \{1,\ldots,d\}}
(-1)^{\#H} M(m+\onebold_H)\,.
\end{equation}
For $S\bigl(Y(m);z,z)$ we have an estimate with product structure.
Lemma~\ref{lem-Mest} estimates $M(n)$ also by a product over the
places, with the exception of $n=\onebold$ in some cases. To handle
that exception we take subsums $T_F$ of~\eqref{sumSM} characterized
by $m_j\geq 2$ if and only if $j\in F$, with $F$ running over the
nonempty subsets of~$\{1,\ldots,d\}$.
\begin{align}
\nonumber
T_F&\;\ll_{\Gm E,l,c}\; \sum_{m\in \NN^d\,,\; m_j\geq 2\Leftrightarrow
j\in F} \max\bigl( m_j^2,\y_j(z)\, m_j\bigr)\\
\nonumber
&\qquad\hbox{} \cdot
 \sum_{H\subset \{1,\ldots,d\}} (-1)^{\#H} \prod_{j\in H }
f_{l,j}(m_j+1) \; \prod_{j\not\in H} f_{l,j}(m_j)
\displaybreak[0]\\
\label{TFprod}
&\= \prod_{j\in F} \sum_{m_j\geq 2} \max\bigl( m_j^2,\y_j(z)\,
m_j\bigr) \bigl( f_{l,j}(m_j) - f_{l,j}(m_j+1)\bigr) \\
\nonumber
&\qquad\hbox{} \cdot
 \prod_{j\not\in F} \max\bigl( 1, \y_j(z)
\bigr) \bigl( f_{l,j}(1) - f_{l,j}(2)\bigr)\,.
\end{align}
For each $j$ and each $m_j$ we will choose a value of $l$ in
$f_{l,j}(m_j)$. Since the implicit constant in the estimate depends
on $l$, we have to choose the $l$ from a finite set in $\NN$. For
``small values'' of $m_j\geq 2$ we take $l=1$, and $l=3$ for
``large'' values. We will determine the boundary between small and large
in a moment. If $m_j=1$, $f_{j,l}(1)$ does not depend on~$l$.

For $j\in F$, the sum
\[ \sum_{m_j\geq 2} \max\bigl( m_j^2,\y_j(z)\, m_j\bigr) \bigl(
f_{l,j}(m_j) - f_{l,j}(m_j+1)\bigr)\]
has two critical values of $m_j$. The first occurs near
$m_j \approx \y_j(z)$. The other one, which we denote by $A_j$ occurs
where $f_{1,j}(m_j) \approx f_{3,j}(m_j)$. We take $l=3$ for $m \geq
A_j+1$ and $l=1$ for $2\leq m_j \leq A_j$, with
$A_j = \bigl[V_j/Y_j\bigr] =  \bigl[V_j^{1-\th}\bigr]$ if $j\in Q$
and $A_j=\bigl[Y_E^{-1}\bigr]$ for $j\in E$. The choices of the
parameters in \S\ref{sect-mp} and~\S\ref{sect-ap} is such that $A_j$
is a large quantity in all cases.

With the notations $P_1=V_j^{1/2}$, $P_2=V_j^{\frac 52-2\th}$ if
$j\in Q$, and $P_1=1$, $P_2=Y_E^{-2}$ if $j\in E$, the sum for
$j\in F$ is equal to
\begin{equation}\label{smex}
\begin{aligned}
\sum_{2\leq m \leq A_j-1}& \max\bigl(m^2, \y_j(z)\, m\bigr) \bigl(
m^{-\frac32} -
(m+1)^{-\frac32} \bigr) P_1\\
&\hbox{}+ \max\bigl( A_j^2,\y_j(z)\, A_j\bigr) \bigl( A_j^{-\frac32}
P_1 -
(A_j+1)^{-\frac 72} P_2\bigr)\\
&\hbox{}+ \sum_{m\geq A_j+1} \max\bigl( m^2, \y_j(z)\, m\bigr) \bigl(
m^{-\frac72}- (m+1)^{-\frac 72} \bigr) P_2\,.
\end{aligned}
\end{equation}
If $\y_j(z) \geq A_j$, the first of the sums is estimated by
\[ \y_j(z) \sum_{2\leq m \leq A_j-1} m^{-\frac 32} P_1 \;\ll\; \y_j(z)
\, P_1\,.\]
If $\y_j(z)< A_j$ we split up the first sum, and obtain
\begin{align*}
\sum_{2\leq m <\y_j(z)} m^{-\frac 32} P_1 + \sum_{\y_j(z)<m\leq A_j-1}
m^{-\frac12 } P_1 \;\ll\; \y_j(z)\, P_1 + A_j^{\frac12}P_1 \,.
\end{align*}
For the last sum we obtain if $\y_j(z) \geq A_j$
\begin{align*}
\sum_{A_j+1\leq m < \y_j(z)} \y_j(z) m^{-\frac 72} P_2 + \sum_{m\geq
\y_j(z)} m^{-\frac 52} P_2 \;\ll\; \y_j(z) \, A_j^{-\frac 52}P_2 +
\y_j(z)^{-\frac 32} P_2\,,
\end{align*}
and if $\y_j(z)< A_j$
\[ \sum_{m\geq A_j+1} m^{-\frac 52} P_2 \;\ll\; A_j^{-\frac 32}P_2\,.
\]
The transitional middle term in~\eqref{smex} can be estimated by
\[ \max\bigl( A_j, \y_j(z)\bigr) A_j^{-\frac12}P_1\,.\]

In total we get for $\y_j(z) \geq  A_j$
\begin{align*}
&\;\ll\; \y_j(z) \bigl(P_1 + A_j^{-\frac 52} P_2\bigr) +
\y_j(z)^{-\frac 32} P_2 + \y_j(z) A_j^{-\frac12}P_1\\
&\;\ll\; \begin{cases}
\y_j(z) V_j^{\frac12} + \y_j(z)^{-\frac 32} V_j^{\frac 52-2\th}
&\text{ if }j\in Q\,,\\
\y_j(z)+\y_j(z)^{-\frac32} Y_E^{-2}&\text{ if }j\in E\,,
\end{cases}
\end{align*}
and for $\y_j(z)<A_j$
\begin{align*} &\;\ll\; \y_j(z) P_1
+ A_j^{\frac12}P_1 + A_j^{-\frac32}P_2 + A_j^{\frac12}P_1\\
&\;\ll\; \begin{cases}
\y_j(z)\, V_j^{\frac12} + V_j^{1-\th/2}
&\text{ if }j\in Q\,,\\
\y_j(z) + Y_E^{-\frac12}
&\text{ if }j\in E\,.
\end{cases}
\end{align*}
In the case $\y_j(z) \geq A_j$ we use that
$\y_j(z)^{-\frac32} V_j^{\frac 52-\th} \leq V_j^{1-\frac12\th}$ in
the case $j\in Q$ and $\y_j(z)^{-\frac 32} Y_E^2 \leq Y_E^{-\frac12}$
for $j\in E$, to get the following bound for the quantity
in~\eqref{smex}:
\begin{equation}
\begin{aligned}
&\;\ll\; \begin{cases}
\y_j(z)\, V_j^{\frac12} + V_j^{1-\frac12\th}
&\text{ if }j\in Q\,,\\
\y_j(z) + Y_E^{-\frac12}
&\text{ if }j\in E\,,
\end{cases}\\
&\;\ll\;
 \begin{cases}
 n_j\bigl( V_j^{\frac12-\frac12\th},z\bigr)\, V_j^{1-\frac12\th}
&\text{ if }j\in Q\,,\\
n_j\bigl(Y_E^{-\frac12} ,z\bigr)\, Y_E^{-\frac12}
&\text{ if }j\in E\,.
\end{cases}
\end{aligned}
\end{equation}
So this estimates the factors with $j\in F$ in~\eqref{TFprod}.

We estimate the factors for $j\not\in F$ by $n_j(z)f_{l,j}(1) $. By
Lemma~\ref{lem-Mest} we get
\begin{equation}
\begin{aligned}
T_F &\;\ll_{E,c}\; \prod_{j\in F \cap Q}
n_j\bigl(V_j^{\frac12-\frac12\th},z\bigr)\, V_j^{1-\frac12\th} \;
\prod_{j\in Q \setminus F} V_j^{\max(c,\hat\tau)+\frac12}
\\
&\qquad\hbox{} \cdot
\prod_{j\in F \cap E} n_j\bigl(Y_E^{-\frac12},z\bigr)\,
Y_E^{-\frac12}\; \prod_{j\in E \setminus F} 1\,.
\end{aligned}
\end{equation}
Since we have already $n(z)$ in the error term in Lemma~\ref{lem-KE},
it seems sensible to replace $n_j(\ast,z) $ by $n_j(1,z)$ in these
estimates. So we put $n(z)$ in front, and remove the $n_j(\cdots)$
from the products.

The next step is to determine which non-empty $F\subset\{1,\ldots,d\}$
has the maximal contribution. The factors for $j\in E$ are maximal if
$j\in F$. So we consider $F\supset E$. The factors for $j\in Q$ are
maximal for $j\not\in F$ if $\th\geq 1-2\max(c,\hat\tau)$, and
maximal for $j\in F$ otherwise. If $E=\emptyset$, we have to put one
place in~$F$ anyhow, which gives the maximal contribution if
$V_j=\vmin$. We find the following maximal value:
\begin{equation}
\begin{cases}
n(z)\, Y_E^{-\frac12\#E} \prod_{j\in Q} V_j^{\max(c,\hat
\tau)+\frac12}
&\text{ if }\th\geq 1-2\max(c,\hat\tau)\text{ and }E\neq
\emptyset\,,\\
n(z) \vmin^{\frac12-\max(c,\hat\tau)-\frac12\th} \prod_{j\in Q}
V_j^{\max(c,\hat \tau)+\frac12}
&\text{ if }\th\geq 1-2\max(c,\hat \tau)\text{ and }\#Q=d\,,\\
n(z) \, Y_E^{-\frac12\#E} \prod_{j\in Q\cap F} V_j^{1-\frac12\th}
&\text{ if }\th\leq 1-2\max(c,\hat\tau)
\,.
\end{cases}
\end{equation}
In the latter case, the maximum is attained for $F=\{1,\ldots,d\}$,
and in the former case for $F=E$, if $E\neq\emptyset$. If
$E=\emptyset$, one place has to be in $F$, and $j$ such that
$V_j=\vmin$ gives the maximal value.

Finally we have to consider the term with $m=\onebold$
in~\eqref{sumSM}. It is estimated by
\begin{equation}
S\bigl( Y(\onebold);z,z) M(\onebold) \ll n(z)\cdot
\begin{cases}
\prod_{j\in Q} V_j^{\max(c,\hat\tau)+\frac12}
&\text{ if }E\neq \emptyset\,,\\
\vmin^{c-\max(c,\hat\tau)}\, \prod_{j\in Q}
V_j^{\max(c,\hat\tau)+\frac12}
&\text{ if }E=\emptyset\,.
\end{cases}
\end{equation}
If $E\neq \emptyset$ or if $\th\leq 1-2\max(c,\hat\tau)$ this is
absorbed in the term that we have already obtained. In the case
$E=\emptyset$ and $\th \geq 1-2\max(c,\hat \tau)$ we have to compare
the factors $\vmin^{\frac12-\max(c,\hat\tau)-\frac12\th}$ and
$\vmin^{c-\max(c,\hat\tau)}$. If $\hat\tau=0$, we have
$\max(c,\hat\tau)=c$, in which case the latter factor, $\vmin^0=1$,
is the largest. In the remaining case there is another transition
point at $\th=1-2c$.

This leads to the statements in the lemma. We resist the temptation to
simplify the lemma by choosing $c< \frac{1-\th}2$. That would cause a
dependence of the implicit constant in the estimates on the auxiliary
parameter~$\th$.
\end{proof}

\subsection{Difference between sums with sharp and smooth
bounds}\label{sect-shsm}
We have obtained $K(z,z)=\Ecal(U,V;z)+\mathit{Err}_1 + \mathit{Err}_2$
for the sum $K(z,z)$ in~\eqref{Kdef1}, the explicit term
$\Ecal(U,V;z)$ in~\eqref{Ecal-def}, with error terms $\mathit{Err}_1$
estimated in Lemma~\ref{lem-KE} and $\mathit{Err}_2$ in
Lemma~\ref{lem-sKne}. The sum $K(z,z)$ depends on the choice of the
local test functions $k_j$ as indicated in~\eqref{k-prop}. In
particular the estimate is valid for the sum $K^+(z,z)$ based on test
functions with $k_j^+=1$ on $[U_j,V_j]$ for all~$j$, and also for the
sum $K^-(z,z)$ built with $\supp(k_j^-)\subset [U_j,V_k]$. Since the
characteristic function $\ch$ of $\prod_j[U_j,V_j)$ satisfies
$\prod_j k^-_j \leq \ch \leq \prod_j k^+_j$, we have
$K^-(z,z)\leq \cnt(U,V;z) \leq K^+(z,z)$. Thus we have also
\begin{equation}
\cnt(U,V;z) \= \Ecal(U,V;z) +\mathit{Err}_1 + \mathit{Err}_2\,,
\end{equation}
with error terms satisfying the estimates in Lemmas~\ref{lem-KE}
and~\ref{lem-sKne}.

\subsection{Asymptotic estimate, case $E=\emptyset$}\label{sect-Qa}
First we choose the auxiliary parameters in the case $E=\emptyset$.
This leads to the asymptotic result in Theorem~\ref{thm-Qa}, of which
Theorem~\ref{thm-Qai} is a special case.

The auxiliary parameter $\th\in(0,1)$ has to be adapted to the $V_j$
to get the minimal value of the bound
\begin{equation}
\begin{aligned}
&\ll_{\Gm,c}\; n(z) \vmin^{\th-1}\prod_j V_j\\
&\qquad\hbox{}
 + n(z) \cdot
\begin{cases}
\prod_j V_j^{1-\frac12\th}&\text{ if }\th\leq 1-2\max(c,\hat\tau)\,,\\
\vmin^{\max\bigl( \frac{1-\th}2,c \bigr)-\hat\tau} \prod_j
V_j^{\frac12+\max(c,\hat\tau)}
&\text{ if }\th\geq 1-2\max(c,\hat\tau)\,.
\end{cases}
\end{aligned}
\end{equation}
The parameter $c\in \bigl(0,\frac12\bigr)$ is allowed to depend
on~$\Gm$, but not on the~$V_j$. We have assumed that $0<c<\hat\tau$
if $\hat\tau>0$. The $V_j$ influence the estimate by their product
$\prod_j V_j$, which tends to~$\infty$. We have prescribed that the
minimal $V_j$ is coupled to the product by
$\vmin^{\hat q}=\prod_j V_j$, with $\hat q\geq d$.
(See~\eqref{hatq}.)
Expressing the logarithm of the quantity to consider in terms of
$\log\vmin$, we arrive at the following quantity to minimize:
\begin{equation}\label{lm1}
 \begin{cases}
\max\bigl( \th-1+\hat q ,\; \hat q-\frac12\hat q \th \bigr)&\text{ if
}\th \leq 1-2\max(c,\hat\tau)\,,\\
\max\bigl( \th-1+\hat q ,\;\frac{\hat q}2 + (\hat
q-1)\hat\tau+\max\bigl( \frac{1-\th}2,c\bigr) \bigr)&\text{ if }\th
\geq 1-2\hat\tau\,,\; \hat\tau>0\,,\\
\max\bigl( \th-1+\hat q ,\;\hat q\bigl(\frac12+c \bigr) +\max\bigl(
\frac{1-\th}2,c\bigr) \bigr)&\text{ if }\th \geq 1-2c\,,\;
\hat\tau=0\,.
\end{cases}
\end{equation}
We choose $0<c < \frac1{18} $ in addition to the requirement that
$c<\hat\tau$ if $\hat\tau>0$.

If $\hat\tau=0$ the value $\th_1$ for which
$\th_1-1+\hat q=\hat q-\frac12\hat q\th_1$ is
$\th_1=\frac2{\hat q+2}$. Since $\hat q\geq d\geq 1$, we have
$\th_1\leq \frac23 <1-2c$. Hence this is the optimal choice.

For $\hat\tau>0$ we have $\th_1=\frac2{\hat q+2}$ and
$\th_2=1-\frac13\hat q+\frac23(\hat q-\nobreak 1)\hat\tau$, for the
intersections of the graph of $\th\mapsto \th-1+\hat q$ with
respectively, $\th\mapsto \hat q-\frac12\hat q\th$ and
$\th \mapsto \frac12\hat q+(\hat q-\nobreak 1)\hat\tau+\frac{1-\th}2$.
If $\hat\tau \leq \frac{\hat q}{2(\hat q+2)}$, then
$\th_1 \leq 1-2\hat\tau$ gives the optimal choice. Otherwise,
$\th_2 \geq 1-2\hat\tau$ is optimal, since it is between
$1-2\hat\tau$ and $1-2c$.

This leads to the following optimal bound of the quantity
in~\eqref{lm1}:
\begin{equation} \hat q\frac{\hat q+1}{\hat q+2}\;\text{ if }0\leq
\hat\tau\leq \frac {\hat q}{2(\hat q+2)}\,,\qquad \hat q
\frac{2(\hat\tau+1)}3-\frac{2\hat\tau}3\;\text{ if }\hat \tau \geq
\frac{\hat q}{2(\hat q+2)}\,.
\end{equation}
Now we have chosen $c$ depending only on quantities determined
by~$\Gm$, and we have arrived at the following estimate:
\begin{equation}
\cnt(U,V;z) - \Ecal(U,V;z) \;\ll_\Gm n(z)\cdot
\begin{cases}
\prod_j V_j^{\frac{\hat q+1}{\hat q+2}}
&\text{ if }\hat\tau\leq \frac{\hat q}{2(\hat q+2)}\,,\\
\prod_j V_j^{\frac23(\hat \tau+1)-\frac{2\hat\tau}{3\hat q}}
&\text{ if }\hat\tau\geq \frac{\hat q}{2(\hat q+2)}\,.
\end{cases}
\end{equation}
If there are totally exceptional eigenvalues, the explicit sum
$\Ecal(U,V;z)$ in~\eqref{Ecal-def} contains the corresponding terms,
which are of the size $\oh\bigl( n(z) \prod_jV_j^{\frac12+\hat
\tau}\bigr)$. Since $\tau_{\ell,j}\leq \hat\tau$ for all exceptional
coordinates, these terms are swallowed by the error term obtained for
the case $\hat\tau \leq \frac{\hat q}{2(\hat q+2)}$.

We have thus obtained the following asymptotic result for the counting
function:
\begin{thm}\label{thm-Qa}Let $\Gm$ be an irreducible lattice in
$\PSL_2(\RR)^d$ with $d\in \NN$. Let $\hat\tau$ be the quantity
in~\eqref{tau-hat}, measuring the spectral gap. Denote by
$V_j\geq 1$, $1\leq j \leq d$, large quantities subject to the
condition $\min_j V_j^{\hat q} = \prod_j V_j$ for a fixed number
$\hat q\geq d$. Choose $U_j=0$ or $U_j=V_j$ for each $j=1,\ldots,d$.

Let $z\in \uhp^d$. The number $\cnt(U,V;z)$ of $\gm\in \Gm$ such that
$U_j \leq u\bigl( (\gm z)_j,z_j) \leq V_j$ for all~$j$, with
$u(\cdot,\cdot)$ as in~\eqref{u-d}, satisfies
\begin{equation}
\cnt(U,V;z) \= \frac{(4\pi)^d}{\vol( \Gm\backslash\uhp^d)}
\prod_{j=1}^d (V_j-U_j) + \oh_\Gm \biggl( n(z) \prod_{j=1}^d V_j
^{(\hat q+1)/(\hat q+2)} \biggr)
\end{equation}
if $\hat\tau \;\leq\; \frac{\hat q}{2(\hat q+2)}$, and
\begin{equation}
\begin{aligned}
\cnt(U,V;z) &\= \frac{(4\pi)^d}{\vol( \Gm\backslash\uhp^d)}
\prod_{j=1}^d (V_j-U_j) + \sum_{\ell\geq 1\,,\; \forall_j\;
\tau_{\ell,j} \in
(0,\frac12)} \bigl|\ps_\ell(z)\bigr|^2 \\
&
\qquad\qquad\hbox{} \cdot
\prod_{j=1}^d \frac{\sqrt\pi \, 2^{1+2\tau_{\ell,j}}
\,\Gf(\tau_{\ell,j})}{\Gf\bigl(\frac32+\tau_{\ell,j}\bigr)}\bigl(
V_j^{\frac12+\tau_{\ell,j}} - U_j^{\frac12+\tau_{\ell,j}}\bigr)
\quad\hbox{ }
\\
&\qquad\hbox{}
+ \oh_{\Gm}\biggl(n(z) \prod_{j=1}^d V_j ^{\frac23(\hat
\tau+1)-2\hat\tau/3\hat q } \biggr)
\end{aligned}
\end{equation}
if $\hat\tau \;\geq \; \frac{\hat q}{2(\hat q+2)}$. The factor $n(z)$
is as in~\eqref{ndef}.
\end{thm}
Note that even if there is no spectral gap ($\hat\tau=\frac12$) the
error term is still smaller than the main term.

In the special case when all $V_j$ are equal to the same quantity~$V$
we have $\hat q=d$. The relation~\eqref{u-d} implies that the
condition $\dist\bigl( (\gm z)_j,z_j \bigr)\leq T$ is equivalent to
$u\bigl( (\gm z)_j,
z_j \bigr)\leq \frac14 e^T \bigl(1+\nobreak \oh(e^{-T} )\bigr)$.
Thus, we obtain Theorem~\ref{thm-Qai}.

\subsection{Asymptotic estimate, case
$E\neq \emptyset$}\label{sect-QEa}We turn to the case when both parts
of the partition $\{1,\ldots,d\}=Q\sqcup E$ are non-empty.

We have obtained the following estimate for the error terms:
\begin{equation}{}\;\ll_{\Gm,c,E}\; n(z) \bigl(
\vmin^{\th-1}+Y_E\bigr) \prod_{j\in Q} V_j
+ n(z) Y_E^{-\frac12\#E} \prod_{j\in Q} V_j ^{\max(1-\frac12 \th,
\frac 12+\hat\tau,\frac12+c)}\,.
\end{equation}
We try to choose $\th$ and $Y_E$ optimally, depending on the~$V_j$,
$j\in Q$. The parameter $c\in \bigl(0,\frac12\bigr)$ satisfies
$c<\hat\tau$ if $\hat \tau>0$, and can be further adapted to the
situation, but is not allowed to depend on the~$V_j$ with $j\in Q$.

We take $x=\log\vmin$ as the large variable. Then
$\prod_{j\in Q} V_j = e^{\hat q x}$, with $\hat q\geq q$ fixed. We
assume that $Y_E = \vmin^{-\eta}$ with $\eta>0$ is a sensible choice.
To simplify the formulas we work for the moment with the notations
$e=\#E$, $m=\max(c,\hat\tau)$. So $0<m\leq \frac12$. We try to choose
$\th\in [0,1]$ and $\eta\geq 0$ such that the following quantity is
minimal:
\begin{equation}
M(\eta,\th) \= \max\biggl( \th-1+\hat q,\, \hat q-\eta, \,\frac
e2\eta+\hat q\biggl(-\frac\th2\biggr), \,\frac e2\eta+\hat
q\biggl(\frac12+m\biggr) \biggr)\,.
\end{equation}
We allow for the moment $\th$ and $\eta$ to assume boundary values. If
we end up with an optimal choice on the boundary, we will see how to
handle the problem of satisfying the conditions in~\S\ref{sect-ap}.

The lines $\th+\eta=1$ and $\th=1-2m$ give four subsets of the region
in which $(\eta,\th)$ varies:
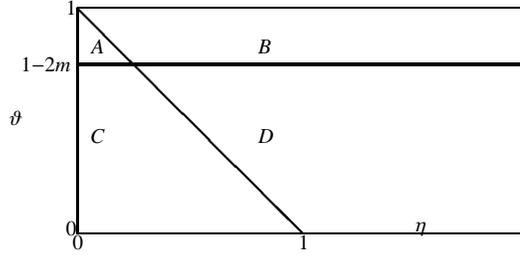
\[\setlength\unitlength{3cm}
\begin{picture}(2,1)
\put(0,1){\line(1,0){2}}
\put(0,0){\line(1,0){2}}
\put(0,0){\line(0,1){1}}
\put(-.05,.97){$\scriptstyle 1$}
\put(-.25,.72){$\scriptstyle 1-2m$}
\put(-.05,-.01){$\scriptstyle 0$}
\put(-.3,.48){$\scriptstyle \th$}
\put(-.02,-.07){$\scriptstyle 0$}
\put(.98,-.07){$\scriptstyle 1$}
\put(1.5,.01){$\scriptstyle \eta$}
\put(.06,.8){$\scriptstyle A$}
\put(.8,.8){$\scriptstyle B$}
\put(.06,.4){$\scriptstyle C$}
\put(.8,.4){$\scriptstyle D$}
\thicklines
\put(0,.75){\line(1,0){2}}
\put(0,1){\line(1,-1){1}}
\end{picture}
\]
We have
\[M(\eta,\th)\=
\begin{cases}
\max\bigl( \hat q-\eta,\,\frac e2\eta+\hat q(\frac12+m) \bigr)&\text{
on }A\,,\\
\max\bigl( \hat q-1+\th,\, \frac e2\eta+\hat q(\frac12+m)
\bigr)&\text{ on }B\,\\
\max\bigl(\hat q-\eta,\, \frac e2\eta+\hat q(1-\frac \th2)
\bigr)&\text{ on }C\,,
\\
\max\bigl( \hat q-1+\th,\, \frac e2\eta+\hat q(1-\frac \th2)
\bigr)&\text{ on }D\,.
\end{cases}\]
On $B$ and~$D$ these expressions for $M(\eta,\th)$ contain $\eta$ only
once. So in the search for optimal values we can take $\eta$ minimal
in these cases. This brings them into the cases $C$ and $D$,
respectively. In region~$C$, the variable $\th$ occurs only once. So
it makes sense to take it optimal, {\sl i.e.}, $\th=1-2m$ if
$0\leq \eta\leq 2m$ and $\th=1-\eta$ if $2m\leq \eta\leq 1$. This
reduces our search for the optimum to the lines $E$ and~$F$ in the
following figure.
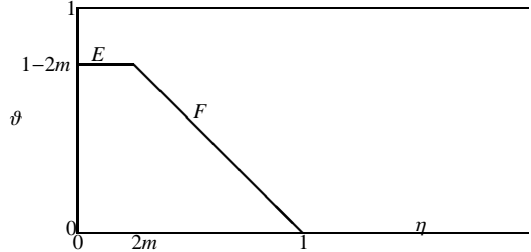
\[\setlength\unitlength{3cm}
\begin{picture}(2,1)
\put(0,1){\line(1,0){2}}
\put(0,0){\line(1,0){2}}
\put(0,0){\line(0,1){1}}
\put(-.05,.97){$\scriptstyle 1$}
\put(-.25,.72){$\scriptstyle 1-2m$}
\put(-.05,-.01){$\scriptstyle 0$}
\put(-.3,.48){$\scriptstyle \th$}
\put(-.02,-.07){$\scriptstyle 0$}
\put(.98,-.07){$\scriptstyle 1$}
\put(.24,-.07){$\scriptstyle 2m$}
\put(1.5,.01){$\scriptstyle \eta$}
\put(.06,.765){$\scriptstyle E$}
\put(.51,.51){$\scriptstyle F$}
\thicklines
\put(0,.75){\line(1,0){.25}}
\put(.25,.75){\line(1,-1){.75}}
\end{picture}
\]

Thus, we are left with a one-dimensional problem: Find the minimum for
$0\leq \eta\leq 1$ of the maximum of the two functions
$\al(\eta) = \hat q-\eta$ and \[ \bt(\eta) \= \begin{cases} \frac
e2\eta+\hat q\bigl( \frac12+m\bigr)&\text{ for }0 \leq \eta\leq
2m,,\\
\frac12\bigl(e+\hat q\bigr)+\frac12\hat q&\text{ for }2m\leq \eta\leq
1\,.
\end{cases}
\]
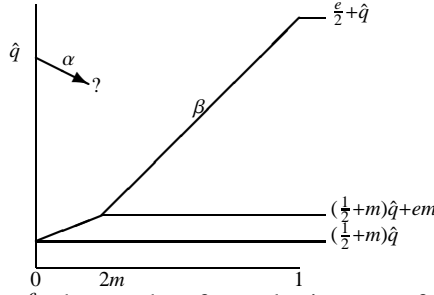
\[ \setlength\unitlength{3.5cm}
\begin{picture}(1.2,1)
\put(0,0){\line(1,0){1}}
\put(0,0){\line(0,1){1}}
\put(-.02,-.07){$\scriptstyle 0$}
\put(.98,-.07){$\scriptstyle 1$}
\put(.24,-.07){$\scriptstyle 2m$}
\put(0,.1){\line(1,0){1.1}}
\put(.25,.2){\line(1,0){.85}}
\put(1,.95){\line(1,0){.1}}
\put(.6,.59){$\scriptstyle\bt$}
\put(1.12,.95){$\scriptstyle \frac e2+\hat q$}
\put(1.12,.2){$\scriptstyle (\frac12+m)\hat q+em$}
\put(1.12,.1){$\scriptstyle (\frac12+m)\hat q$}
\put(.1,.77){$\scriptstyle\al$}
\put(.21,.67){$\scriptstyle ?$}
\put(-.1,.8){$\scriptstyle \hat q$}
\thicklines
\put(0,.1){\line(5,2){.25}}
\put(.25,.2){\line(1,1){.75}}
\put(0,.8){\vector(2,-1){.2}}
\end{picture}
\]
Since $\bigl(\frac12+m\bigr)\hat q \leq \hat q \leq \hat q+\frac e2$,
the graphs of $\al$ and~$\bt$ intersect for some value of $\eta$ in
$[0,1]$. This gives the value of the optimum we look for. This leads
to the optimal value
\begin{alignat*}2 M\biggl( \frac{ \hat q}{\hat q+e+2}, \frac{e+2}{\hat
q+e+2} \biggr) &\= \frac{\hat q+e+1}{\hat q+e+2}\hat q&\quad&\text{
if } m \leq \frac{\hat q}{2(\hat q+e+2)}\,,\\
M\biggl( \frac{(1-2m)\hat q}{e+2},1-2m\biggr) &\=
\frac{e+1+2m}{e+2}\hat q&&\text{ if }m\geq \frac{\hat q}{2(\hat
q+e+2)}\,.
\end{alignat*}

We note that if $m=\max(x,\hat \tau)=\frac12$ (no spectral gap), the
optimal value is at the boundary point $(0,1)$, which we did not want
to use. However, this optimal value is $\hat q\geq q$, hence it
provides an error term that swallows the main term. So we assume that
$\hat \tau<\frac12$ from this point on.

If $\hat \tau>0$, we can just replace $m$ by~$\hat \tau$ in the
results. If $\tau=0$ we take $0< c< \frac{\#Q}{2(d+2)}$. This depends
on $d$, hence on~$\Gm$, and on the partition in $Q$ and~$E$. We now
have
\[ \frac{\hat q}{2(\hat q+e+2)}\= \frac12-\frac{2+e}{2(\hat q+2+e)}
\geq \frac12-\frac{2+e}{2(2+e+\#Q)} \= \frac{\#Q}{2(d+e)}>c=m\,.\]
So we can apply the estimate for $m\leq \frac{\hat q}{2(\hat q+e+2)}$.

Thus, we obtain the following estimates for the error terms:
\begin{equation}\label{Er-Ene}
\begin{aligned}
&\oh_{\Gm,E}\biggl( n(z)\, \prod_{j\in Q} V_j ^{(\hat q+1+\#E)/(\hat
q+2+\#E)} \biggr)&&\text{ if }\hat \tau \leq \frac {\hat q}{2(\hat
q+2+\#E)}\,,\\
&\oh_{\Gm,E} \biggl( n(z)\, \prod_{j\in Q} V_j ^{(1+2\hat
\tau+\#E)/(2+\#E)}\biggr)&&\text{ if }\frac{\hat q}{2(\hat
q+2+\#E)}\leq \hat \tau<\frac12\,.
\end{aligned}
\end{equation}

Each exceptional term in $\Ecal(U,V;z)$ in~\eqref{Ecal-def}
contributes at most $\prod_{j\in Q}V_j ^{\frac12+\hat \tau}$ and is
absorbed by the error term in~\eqref{Er-Ene}. We are left with the
term corresponding to the constant function. See \eqref{eta-b} in
Lemma~\ref{lem-eta-est} for its simple form.

Thus we have obtained the following asymptotic result:
\begin{thm}\label{thm-QEa}
Let $\Gm$ be an irreducible lattice in $\PSL_2(\RR)$ with $d\geq2$.
Suppose that the quantity $\hat\tau$ in~\eqref{tau-hat} satisfies
$\hat\tau(\Gm)<\frac12$. We partition the set $\{1,\ldots,d\}$ into
two disjoint non-empty subsets $Q$ and~$E$. For each $j\in E$ we fix
a bounded interval $[U_,V_j]\subset [0,\infty)$. For each $j\in Q$,
let $V_j\rightarrow\infty$, under the assumption that
$\min_{j\in Q} V_j^{\hat q} = \prod_{j\in Q}V_j$ for some fixed real
number $\hat q\geq \#Q$. Also, choose $U_j=0$ or $U_j=\frac12V_j$ for
$j\in Q$.

Let $z\in \uhp$. The number $\cnt(U,V;z)$ in~\eqref{cntdef} satisfies
\begin{equation}
\begin{aligned}
\cnt(U,V&;z) \= \frac{(4\pi)^d}{\vol(\Gm\backslash\uhp^d)}
\prod_{j=1}^d
(V_j-U_j) \\
&\hbox{} + \begin{cases}
\oh_{\Gm,E}\biggl( n(z)\, \prod_{j\in Q} V_j ^{(\hat q+1+\#E)/(\hat
q+2+\#E)} \biggr)&\text{ if }\hat \tau \leq \frac {\hat q}{2(\hat
q+2+\#E)}\,,\\
\oh_{\Gm,E} \biggl( n(z)\, \prod_{j\in Q} V_j ^{(1+2\hat
\tau+\#E)/(2+\#E)}\biggr)&\text{ if }\frac{\hat q}{2(\hat
q+2+\#E)}\leq \hat \tau<\frac12\,.
\end{cases}
\end{aligned}
\end{equation}

The implicit constants in the estimates depend on the discrete
group~$\Gm$, on the partition $\{1,\ldots,d\}=Q\sqcup E$ and on the
choice of the intervals $[U_j,V_j)$ for $j\in E$.
\end{thm}

We note that the presence of totally exceptional eigenvalues for~$\Gm$
has no explicit influence on this asymptotic formula. The size of he
spectral gap does influence the quality of the error term only if
it is larger than $\frac{\hat q}{2(\hat q+2+\#E)}$.

Like we did in the previous subsection we take all $V_j$ with $j\in Q$
equal to $V$, and all $U_j$ for $j\in Q$ equal to $0$. With the
relation in~\eqref{u-d} between
$u\bigl( (\gm z)_j,z_j\bigr) \in [U_j,V_j)$ and the hyperbolic
distance $\dist\bigl( (\gm z)_j,z_j\bigr) \in [A_j,B_j)$, the main
term takes the form
\[ \frac{4^{\#E}\pi^d}{\vol(\Gm\backslash\uhp^d)} \, e^{dT} \,
\prod_{j\in E} \frac{ e^{B_j}+e^{-B_j}-e^{A_j}-e^{-A_j}}4\,,\]
in the notations of Theorem~\ref{thm-QEai}, which leads to the main
term in the asymptotic formula in that theorem. For the error terms
we use that for equal $V_j$ for $j\in Q$, the parameter $\hat q$ is
equal to~$\#Q$.

\section{Estimates of Selberg transforms}\label{sect-eSt}
Here we collect and prove the estimates of Selberg transforms that we
have used. This can be done factor by factor. So in this section we
work on~$\uhp$, we do not use an index~$j$ and we denote real numbers
by $U $ and~$V$.
\subsection{Integral representations}\label{sect-gc}
The results we need are given in Lemmas~\ref{lem-eta-est},
\ref{lem-se1} and~\ref{lem-se2}. To derive these lemmas we start with
an arbitrary measurable compactly supported function $k$ on
$[0,\infty)$ with values in $[0,1]$. By the definitions
in~\S\ref{sect-St} we have the following integral representation for
the Selberg transform~$h$ of~$k$:
\begin{align}
\nonumber
h(\tau) &\= 2\int_{r=-\infty}^\infty e^{r\tau} \int_{u=\sinh^2
r/2}^\infty k(u) \, \frac{du}{\sqrt{u-\sinh^2r/2}}\, dr\\
\label{h-ru}
&\= 4\int_0^\infty \cosh r\tau\; \int_{u=\sinh^2r/2}^\infty k(u)
\,\frac{du}{\sqrt{u-\sinh^2r/2}}\, dr\,.
\end{align}
The inner integral gives a non-negative compactly supported function.
So $h(0) \geq |h(it)|$ for $t\in \RR$, which gives
\begin{equation}
\label{hr}
|h(it)|\;\leq \; h(0)\quad(t\in \RR)\,,\qquad \tau\mapsto
h(\tau)\text{ is increasing on
}\bigl[0,\txtfrac12\bigr]\,.\end{equation}
(For the latter we use that $\tau\mapsto \cosh r\tau$ is increasing.)

We interchange the order in the double integration, and obtain
\begin{equation}\label{h-ur}
h(\tau)
\= 2\int_{u=0}^\infty k(u) \int_{r\in\RR\,,\; \sinh^2 r/2\leq u}
e^{r\tau}\, \bigl( u - \sinh^2 r/2\bigr)^{-1/2}\, dr\, du\,.
\end{equation}
Following the approach in \cite{Pa75}, we define $x=x(u)\geq 1$ by
$2+4u=x+x^{-1}$, hence $x(u)=1+2u+2\sqrt{u+u^2}$. The condition
$\sinh^2\frac r2\leq u$ amounts to $|r|\leq \log x$. With the
substitution $e^r = x^{-1}\bigl( 1+\nobreak y(x^2-\nobreak1)\bigr)$
the inner integral equals
\begin{align}
\nonumber
&\= \int_{-\log x}^{\log x} e^{r\tau} \bigl(
(x+x^{-1}-e^r-e^{-r})/4\bigr)^{-1/2}\, dr
\\
\label{i-int}
&\= 2 \,x^{\frac12-\tau} \int_0^1 \bigl( 1+y(x^2-1)
\bigr)^{\tau-\frac12} \,\bigl( y(1-y)
\bigr)^{-\frac12}\, dy
\\
\label{i-int1}
&\= 2 \pi \, x^{\frac12-\tau}\, \hypgeom\bigl(
\txtfrac12-\tau,\txtfrac12;1;1-x^2\bigr)\,,
\end{align}
by the standard integral representation of the hypergeometric series
in \cite{EMOT}, \S2.1.3, (10). Let us work under the standing
assumption that $0\leq \re\tau\leq\frac12$.

From~\eqref{i-int} we obtain the bound
 $\oh\bigl( x^{\re\tau-\frac12}\bigr)$ for the inner integral
 in~\eqref{h-ur}. Hence if $\supp(k) \subset [A,B]$, then
\begin{equation}\label{hest1}
\begin{aligned}
h(\tau)&\;\ll\; \int_{x(A)}^{x(B)} x^{2\re\tau-\frac12}\, dx \;\ll\;
\frac{x(B)^{\re\tau+\frac12} -
x(A)^{\re\tau+\frac12}}{\re\tau+\frac12}\\
&\= (B-A) \,
\bigl(1+2T+2\sqrt{T^2+T}\bigr)^{\re\tau-\frac12}\qquad\text{with
}A\leq T\leq B\\
&\;\ll\; \begin{cases}
(B-A)\bigl(1+A^{\re\tau-\frac12}\bigr)
&\text{ if }A>0\,,\\
1+B^{\re\tau+\frac12}&\text{ if }A=0\text{ and }B>0\,.
\end{cases}
\end{aligned}
\end{equation}\smallskip

Proceeding with~\eqref{i-int1} we get
\begin{align}\label{h1}
h(t) &\= \pi \int_1^\infty k\bigl((x+x^{-1}-2)/4\bigr) \,
x^{-\frac32-\tau} \,(x^2-1)\\
\nonumber
&\qquad\hbox{} \cdot
 \hypgeom\bigl( \txtfrac12-\tau,\txtfrac12;1;1-x^2\bigr)\, dx
\displaybreak[0]
\\
\label{h2}
&\= \pi \int_1^2 k\bigl((x+x^{-1}-2)/4\bigr) \, x^{-\frac32-\tau}
\,(x^2-1)\\
\nonumber
&\qquad\hbox{} \cdot \hypgeom\bigl(
\txtfrac12-\tau,\txtfrac12;1;1-x^2\bigr)\, dx\\
\nonumber
&\quad\hbox{}
+ \sqrt\pi \int_2^\infty k\bigl((x+x^{-1}-2)/4\bigr)\, x^{-\frac32}
\,(x^2-1)^{\frac12} \\
\nonumber
&\qquad\qquad\hbox{} \cdot
\sum_\pm \frac{\Gf(\pm \tau)}{\Gf\bigl( \frac12\pm \tau\bigr)}\,
x^{\pm \tau} \, \hypgeom\bigl(\txtfrac12,\txtfrac12;1\mp
\tau;\txtfrac1{1-x^2}\bigr)\, dx\,.
\end{align}
(For~\eqref{h2} we have used a Kummer relation, \cite{EMOT}, \S2.9,
(34), (10), (13).)

Let us use \eqref{h2} in the case when $\supp(k)\subset[A,B]$. For the
part of $[A,B]$ corresponding to a subinterval $x\in[1,2]$ we have
the estimates in~\eqref{hest1}. We now estimate the integral over an
interval $[A,B]$ with $x(A)\geq 2$ by
\[ \sum_\pm \biggl|\frac{\Gf(\pm\tau)}{\Gf(\frac12\pm\tau)}\biggr|
\int_{x(A)}^{x(B)} x^{-\frac12\pm \re\tau}
\bigl|\hypgeom\bigl(\txtfrac12,\txtfrac12;1\mp
\tau;\txtfrac1{1-x^2}\bigr)\bigr|\,dx\,.\]
Uniformly for $\tau$ in a compact set $T$ we find an estimate by
\begin{equation}\label{hest2}
\sum_\pm \biggl|\frac{\Gf(\pm\tau)}{\Gf(\frac12\pm\tau)}\biggr| \,
(B-A)\, A^{\pm\re\tau-\frac12}\,.
\end{equation}
Note that this estimate is bad if $\tau$ is near~$0$. To handle a
neighborhood of $\tau=0$, we consider the second integral
in~\eqref{h2} as a holomorphic function of the complex variable
$\tau$. It is holomorphic at~$\tau=0$, since the contributions of
$\Gf(\tau)$ and $\Gf(-\tau)$ cancel each other. If $\supp(k)$ is
contained in $[A,B]$ with $x(A)\geq 2$, we find for $|\tau|=c$ with a
small $c>0$ by the reasoning that led to the estimate
in~\eqref{hest2} a bound
\[ \oh_c\bigl( (B-A) A^{c-\frac12} \bigr)\,.\]
By holomorphy, this bound extends to $|\tau|\leq c$. Thus, if
$\supp(k)\subset [A,B]$ with $x(A)\geq 2$, and if $|\tau|\leq c$,
then
\[ h(\tau) \;\ll_c\; (B-A)\, A^{c-\frac12}\,.\]

For $\tau=it\in i\RR$ with $x\geq 2$ we use that
\[ \bigl|\hypgeom\bigl(\txtfrac12,\txtfrac12;1\pm
it;\txtfrac1{1-x^2}\bigr)\bigr| \;\leq\;
\hypgeom\bigl(\txtfrac12,\txtfrac12;1;\txtfrac1{1-x^2}\bigr)\,.\]
By Stirling's formula we get a bound
$\bigl(1+\nobreak|t|\bigr)^{-1/2} \,(B-\nobreak A)\, A^{-\frac12}$
uniformly on $|t|\geq c$.

In Table~\ref{tab-he} we have combined these results.
\begin{table}
\[ \renewcommand\arraystretch{1.3}
\setlength\arraycolsep{0.1cm}
\begin{array}{|lc|c|c|c|}\hline
&& \text{i) }c\leq \tau\leq\frac12& \text{ii) }|\tau|\leq c&
\text{iii) } \tau=it \\
&&&& \in i\RR\setminus(-c,c)\\ \hline
\text{a)}&1\leq A <B&(B-A)A^{\tau-\frac12}&
(B-A)A^{c-\frac12}& |t|^{-1/2}(B-A)A^{-\frac12} \\ \hline
\text{b)}&0<A<B\leq 1& \multicolumn{3}{|c|}{(B-A)A^{-\frac12}}\\
\hline
\text{c)}&0=A<B\leq 1& \multicolumn{3}{|c|}{B^\frac12}\\\hline\hline
\text{d)}&0<A<B&(B-A)&
(B-A)& (B-A)A^{-\frac12}
\\
&&\hbox{\ } \cdot \max(A^{\tau-\frac12},A^{-\frac12})
&\hbox{\ } \cdot\max(A^{c-\frac12},A^{-\frac12})
&\hbox{\ } \cdot\max(1,|t|^{-\frac12}B) \\ \hline
\text{e)}&0=A<B&1+B^{\tau+\frac12}& 1+B^{c+\frac12}
& 1+|t|^{-\frac12}B
\\ \hline
\end{array}
\]
\caption{Bounds for $h(\tau)$ under the assumption that
$\supp(k)\subset [A,B]$. These bounds depend implicitly on
$c\in \bigl(0,\frac12\bigr)$.}\label{tab-he}
\end{table}
We will use these estimates repeatedly in the proofs of the following
lemmas. In some cases, we shall return to the integral
representations.

\subsection{Lemmas for Selberg transforms} First we consider the
Selberg transform $\eta(U,V;\tau)$ in~\eqref{eta-def} of the
characteristic function~$\ch$ of a bounded interval
$[U,V)\subset [0,\infty)$.
\begin{lem}\label{lem-eta-est}
\begin{align}
\label{eta-a}
\tau\mapsto &\eta(U,V;\tau)\text{ is positive and increasing on
$\bigl[0,\txtfrac12\bigr]$}\,,\\
\label{eta-b}
\eta\bigl(U,V;\txtfrac12\bigr) &\= 4\pi(V-U)\,,\\
\label{eta-c}
\bigl| \eta(U,V;it) \bigr| &\;\leq\; \eta(U,V;0)\qquad(t\in\RR)\,.
\end{align}
Moreover, if $c\in \bigl(0,\frac12\bigr)$ is fixed, then we have
uniformly for $c\leq\tau\leq \frac12$ the estimate
\begin{equation}
\label{eta-as}
\eta(U,V;\tau) \= \sqrt\pi
\frac{2^{2\tau+1}\Gf(\tau)}{\Gf(\frac32+\tau)} \,\bigl(
V^{\tau+\frac12} - U^{\tau+\frac12}\bigr) + \oh_c \bigl(
V^{-\tau+\frac12}\bigr)\qquad (V\rightarrow\infty)\,.
\end{equation}
\end{lem}
If the difference $V-U$ is small, then the $\oh$-term may be larger
than the explicit term in~\eqref{eta-as}.
\begin{proof}We apply the computations in \S\ref{sect-gc} to the
characteristic function~$\ch$ of $[U,V)$. In~\eqref{hr} we find
\eqref{eta-a} and~\eqref{eta-c}. Taking $\tau=\frac12$
in~\eqref{i-int} gives~\eqref{eta-b}.

For~\eqref{eta-as} we need an asymptotic formula, not an estimate. We
use~\eqref{h2}. If $x(U)\geq 2$, we need only the integral over
$[2,\infty)$. Writing $
\hypgeom\bigl( \txtfrac12-\tau,\txtfrac12;1;1-x^2\bigr) = 1+\oh(x^{-2})$
as $x\rightarrow\infty$, we get, uniformly for
$\tau\in\bigl[c,\frac12\bigr]$:
\begin{align*}
\sum_\pm \pi& \frac{\Gf(\pm\tau)}{\Gf(\frac12\pm\tau)}
\int_{x(U)}^{x(V)} x^{\pm\tau-\frac12}
\bigl(1+\oh_c(x^{-2})\bigr)\,dx\\
&\=\sum_\pm\pi \frac{\Gf(\pm\tau)}{\Gf(\frac32\pm\tau)} \bigl(
x(V)^{-\frac12\pm\tau}-x(U)^{-\frac12\pm\tau}
+\oh_c\bigl(x(U)^{\pm\tau-\frac32} \bigr) \bigr)\,.
\end{align*}
For $T\geq 1$ we have $x(T) = 4T+\oh(1)$. So the main term with
$\pm=+$ gives the explicit term in~\eqref{eta-as}. The other terms
give
$\oh(U^{\tau-\frac32})+\oh(V^{\frac12-\tau})=\oh\bigl(V^{\frac12-\tau}\bigr)$.

If $x(U)\leq 2$ we get from $x\in[2,\infty)$ the contribution
\[ \pi \frac{\Gf(+\tau)}{\Gf(\frac32+\tau)} \bigl( V^{\tau+\frac12} -
\oh(1)
\bigr) + \oh(V^{\frac12-\tau})\,.\]
We add to it $\oh\left( u(2)-U \right)= \oh(1)$ from~i)d) in
Table~\ref{tab-he}, and obtain \eqref{eta-as} in this case as well.
\end{proof}

Next we consider functions approximating the characteristic function
of $[U,V)$, satisfying the following conditions.
\begin{equation}\label{kj-cond}
\begin{aligned}
&k\;\in \;C_c^\infty[0,\infty)\,,\qquad 0\leq k \leq 1\,,\\
&\exists_{Y>0}\text{ such that $2Y\leq U$ if $U>0$, $2Y\leq V-U$,
and}\\
&k\=1\text{ on } \begin{cases}
[U+Y,V-Y]&\text{ if }U>0\,,\\
[0,V-Y]&\text{ if }U=0\,,\end{cases}\\
&k\=0\text{ on }\begin{cases}
[0,U-Y]\cup[V+Y,\infty)&\text{ if }U>0\,,\\
[V+Y,\infty)&\text{ if }U=0\,.
\end{cases}
\end{aligned}
\end{equation}

\begin{lem}\label{lem-se1}The Selberg transform $h$ in~\eqref{Se} of a
function $k\in C^\infty[0,\infty)$ satisfying the
conditions~\eqref{kj-cond} has the following properties:
\begin{enumerate}
\item[a)] $4\pi(V-U-2Y) \leq h\bigl(\frac12\bigr)\leq 4\pi(V-U+2Y)$.
\item[b)] If $V<1$, then
\[ \biggl| h(\tau)-h\biggl(\frac12\biggr)\biggr| \;\ll\; V^{3/2}\,
\biggl|\txtfrac12-\tau\biggr|\]
for all $\tau$ with $0\leq \re \tau\leq \frac12$.
\item[c)] For each $c\in \bigl(0,\frac12\bigr)$ the difference with
the Selberg transform $\eta(U,V;\tau)$ of the characteristic function
of $[U,V)$ satisfies the estimate
\[ \eta(U,V;\tau)-h(\tau) \ll_c \begin{cases}Y \, \max\bigl(
V^{\tau-\frac12},U^{-\frac12}\bigr)&\text{ if }U>0\,,\\
Y\max\bigl(V^{-\frac12},V^{\tau-\frac12}\bigr)&\text{ if }U=0\,,
\end{cases}
\]
uniformly in~$\tau\in \bigl[c,\frac12\bigr]$.
\end{enumerate}
\end{lem}
\begin{proof}
Part~a) follows by a comparison of $k$ with the characteristic
functions of the intervals $[U+\nobreak Y,V-\nobreak Y) $ and
$[U,V)$, and an application of \eqref{eta-b} in
Lemma~\ref{lem-eta-est}.

For~b)
we use \eqref{h-alt}, and note that
$y^s-1=s (y-1) \left(1+\x (y-1)\right)^{s-1}$ for some
$\x=\x_{s,y}\in
(0,1)$ to obtain
\begin{align*} \bigl|h(\tau) - h\bigl(\frac12\bigr) \bigr|
&\leq \bigl|\frac12-\tau\bigr| \,\int_\uhp k\bigl(u(z,i)\bigr)\,
|y-1|\; \bigl(1+\x (y-1)\bigr)^{-\frac12-\re \tau} \, d\mu(z)\,.
\end{align*}
For small~$V$ the values of~$y$ that occur in the integral are between
$1-\oh\bigl(\!\sqrt V\bigr)$ and $1+\oh\bigl(\!\sqrt V\bigr)$. Thus
the integral is bounded by
$\oh\bigl(\!\sqrt V\bigr)\, \int_\uhp k\bigl(u(z,i) \bigr)\, d\mu(z)$,
which gives~b).

For~c) we apply the estimate i)d) in Table~\ref{tab-he} to a function
with support in the union of the intervals
$[U-\nobreak Y,U+\nobreak Y]$ and $[V-\nobreak Y,V+\nobreak Y]$. For
$U>0$ we find:
\[ Y \, \max(U^{\tau-\frac12},U^{-\frac12}) + Y\, \max\bigl(
(V-Y)^{\tau-\frac12},(V-Y)^{-\frac12}\bigr)
\;\ll\; Y \, \max(V^{\tau-\frac12},U^{-\frac12})\,. \]
If $U=0$ we have only the contribution of $[V-\nobreak Y,V+Y]$.
\end{proof}

\begin{lem}\label{lem-se2}The Selberg transform $h$ of a function
$k\in C^\infty[0,\infty)$ satisfying the conditions~\eqref{kj-cond}
has the following properties:
\begin{enumerate}
\item[a)] $h(\tau) = \eta(U,V;\tau)+Y$ for
$\tau\in \bigl[0,\frac12\bigr]$.
\item[b)] Let $c\in \bigl(0,\frac12\bigr)$. Then we have, uniformly
for $\tau\in i[-c,c]\cup(0,c]$:
\[ h(\tau) \;\ll_c\;
\begin{cases}
(V-U)U^{c-\frac12}&\text{ if } 1\leq U\leq V\,,\\
V^{c+\frac12}&\text{ if }U=0,\; V\geq 1\,,
\end{cases}
\]
and, without dependence on~$c$:
\[ h(\tau) \;\ll\; V\,.\]
\item[c)] Let $c\in \bigl(0,\frac12\bigr)$ and take $l\in \NN$. Then
we have for each $\dt>0$, uniformly for $t\in \RR\setminus(-c,c)$:
\[ h(it) \;\ll_{c,l}\;
\begin{cases}
Y\, \|k^{(l)}\|_\infty \, \max\bigl( V^{l-\frac12}, U^{-\frac12}\bigr)
\, |t|^{-l-\frac12}
&\text{ if }\dt\leq U<V\,,\\
Y\, \|k^{(l)}\|_\infty \, \max\bigl(V^{-\frac12}, V^{l-\frac12}\bigr)
\, |t|^{-l-\frac12}
&\text{ if }U=0,\,V>\dt\,.
\end{cases}
\]
\end{enumerate}
\end{lem}
We note that we have stated what we need, not the best estimate one
might prove by separating more cases. In the proof we will see that
in~c)
we have to avoid intervals with $U\in (0,\dt)$, to be able to apply an
asymptotic estimate for hypergeometric functions.
\begin{proof}
Part~a) is a direct consequence of \eqref{h-alt}, the inequalities
$0\leq k(u) \leq \ch(u)$, where $\ch$ is the characteristic function
of~$[U,V)$, and Lemma~\ref{lem-eta-est}.

The first estimate in b) can be read off from Table~\ref{tab-he}. The
bound $h(\tau)\ll V$ follows from a) and~\eqref{eta-b} in
Lemma~\ref{lem-eta-est}.\smallskip

For d) we modify the discussion in~\S\ref{sect-gc}. By the smoothness
of~$k$ we find in~\eqref{Se}
\begin{equation}
q(p) \= \frac{(-1)^l\sqrt\pi}{\Gf(l+\frac12)} \int_p^\infty k^{(l)}(u)
\,(u-p)^{l-\frac12}\, du
\end{equation}
for each $l\in \NN$. Proceeding as in \eqref{h-ru},
\eqref{h-ur}--\eqref{i-int1}, \eqref{h1}--\eqref{h2}, we obtain
\begin{align}
\nonumber
h(\tau)&\= \frac{\sqrt\pi(-1)^l}{2^{2l}\Gf(l+\frac12)}
\int_{x=1}^\infty k^{(l)}\bigl( (x+x^{-1}-2)/4\bigr) \int_{y=0}^1
x^{-\frac32-\tau-l}\\
&\qquad\hbox{} \cdot
(x^2-1)^{2l+1}\, \bigl(1+y(x^2-1) \bigr)^{\tau-l-\frac12} \, \bigl(
y(1-y)
\bigr)^{l-\frac12}\, dy\, dx
\displaybreak[0]\\
\nonumber
&\= \frac{\pi (-1)^l}{2^{4l}\;l!} \int_1^\infty k^{(l)}\bigl(
(x+x^{-1}-2)/4\bigr)
 \, x^{-\frac32-\tau-l}\, (x^2-1)^{2l+1}\\
 \nonumber
 &\qquad\qquad\hbox{} \cdot
\hypgeom\bigl( l+\txtfrac12-\tau,l+\txtfrac12;2l+1;1-x^2\bigr)\, dx
 \displaybreak[0]\\
 \label{ht}
&\= \frac{\pi (-1)^l}{2^{4l}\;l!} \int_1^2 k^{(l)}\bigl(
(x+x^{-1}-2)/4\bigr)\, x^{-\frac 52-3l-\tau}\\
\nonumber
&\qquad\qquad\hbox{} \cdot
 (x^2-1)^{2l+1}\,
\hypgeom\bigl(l+\txtfrac12+\tau,l+\txtfrac12;2l+1;1-x^{-2}\bigr)\,
dx\\
\nonumber
&\qquad\hbox{}
+ \frac{\sqrt\pi (-1)^l}{2^{2l}}\int_2^\infty k^{(l)}\bigl(
(x+x^{-1}-2)/4\bigr)\, \sum_\pm
\frac{\Gf(\pm\tau)}{\Gf(l+\frac12\pm\tau)} x^{-\frac32\pm\tau-l}
\\
\nonumber
&\qquad\qquad\hbox{}\cdot (x^2-1)^{l+\frac12}\, \hypgeom\bigl(
\txtfrac12-l,\txtfrac12+l;1\mp\tau;(1-x^2)^{-1}\bigr)\, dx\,.
\end{align}

We apply this for $\tau=it\in i\RR$ with $|t|\geq c$. Consider first
an interval $[U_1,V_1]$ with $1<x(U_1)<x(V_1)\leq 2$. Then we use the
first integral in~\eqref{ht} to get a bound
\[ \;\ll_l\; \|k^{(l)}\|_\infty\; \int_{x(U_1)}^{x(V_1)} \bigl|
\hypgeom\bigl(l+\txtfrac12+it,l+\txtfrac12;2l+1;1-x^{-2}\bigr)\bigr|\,dx\,.\]
For $x\geq 1+\dt>1$ we have by formulas (14) and~(15) in \S2.3.2
of~\cite{EMOT}
\[\hypgeom\bigl(l+\txtfrac12,l+\txtfrac12+it;2l+1;1-x^{-2}\bigr)\;\ll_{l,\dt}\;
|t|^{-\frac12-l}\,. \]
This gives a bound
\begin{align*}\oh_{l,\dt}\bigl(\|k^{(l)}||_\infty& \bigl(
x(V_1)-\nobreak x(U_1) \bigr)\bigr) \\
&\;\ll\; \|k^{(l)}||_\infty\,
(V_1-\nobreak U_1)
\bigl(1+\nobreak\frac1{\sqrt{V_1+V_1^2}+\sqrt{U_1+U_1^2}}\bigr)\\
&\;\ll\;\|k^{(l)}||_\infty\,
(V_1-\nobreak U_1)V_1^{-\frac12}\,.\end{align*}
We stress that the use of $\dt$ is critical for the application of the
asymptotic behavior from {\sl loc.\ cit.} If we allow $x$ to get down
to~$1$ the implicit constant blows up.

For an interval $[U_2,V_2]$ with $x(U_2)\geq 2$ we can use the second
integral in~\eqref{ht}. The hypergeometric series shows that
\[ \bigl| \hypgeom\bigl( \txtfrac12-l,\txtfrac12+l;1\mp
it;(1-x^2)^{-1}\bigr)\bigr| \;\leq \;
\hypgeom\bigl(\txtfrac12+l,\txtfrac12+l;1;(1-x^2)^{-1}\bigr) \,,\]
which is $\oh_l(1)$ for $x\geq 2$. By Stirling's formula we get an
estimate by
\[ \oh_l\biggl( \|k^{(l)}\|_\infty \, |t|^{-l-\frac12}
\int_{x(U_2}^{x(V_2)} x^{-\frac12+l}\, dx \biggr) \;\ll\;
\|k^{(l)}\|_\infty\, |t|^{-\frac12-l}
(V_2-U_2)V_2^{l-\frac12}\,. \]

If $U=0$ we have only to estimate the integral over
$[V-\nobreak Y,V+\nobreak Y]$. This gives the bound
$Y \|k^{(l)}\|_\infty \max(V^{-\frac12},V^{l-\frac12})$. If
$U\geq \dt$, we get from the interval $[U-\nobreak Y,U+\nobreak Y]$
the bound
$Y \, \|k^{(l)}\|_\infty\, \max( U^{l+\frac12} , U^{-\frac12})$.
Together with the bound for the interval
$[V-\nobreak Y,V+\nobreak Y]$ we get the other bound in c) of the
lemma.

\end{proof}

\section{Spectral theory}\label{sect-spf}

\subsection{Spectral expansion} \label{sect-spe}
The pointwise convergence of the spectral expansion of sufficiently
differentiable elements of $L^2(\Gm\backslash\uhp^d)$ in
Theorem~\ref{thm-pspe} is similar to well known facts for the case
$d=1$ ({\sl e.g.}, Theorems~4.7 and~7.4 in~\cite{Iw95}), and for
rank-one Lie groups
(Lemma~2.2 in~\cite{MW92}). We sketch how to obtain the pointwise
convergence in Theorem~\ref{thm-pspe} in the present context.

We consider first the mechanism of the Selberg transform
on~$\uhp$, given in~\S\ref{sect-St}. We replace
$k\in C_c^\infty[0,\infty)$ by
\begin{equation}
r_s(u) \= \frac{\Gf(2s)}{4\pi\Gf(s)} u^{-s} \hypgeom(s,s;2s;-1/u)\,,
\end{equation}
with $\re s>1$. It has a logarithmic singularity at $u=0$ and its
support is not compact. Nevertheless, it determines a kernel function
$(z,w) \mapsto r_s(u(z,w))$ on $\uhp$ such that the corresponding
convolution operator $R_s: f \mapsto R_s f$ is well defined for
bounded $f\in C^\infty(\uhp)$. It is in fact the free space resolvent
on $\uhp$, and satisfies
\begin{equation}
R_s \bigl(\Dt-s+s^2\bigr) f\=f\,.
\end{equation}
See \S1.9 of~\cite{Iw95}. If $a,s\in \CC$ both have real part larger
than~$1$, the difference $r_{s,a}=r_s-r_a$ has no singularity at
$u=0$, and has the Selberg transform
\begin{equation}\label{hsa}
h_{s,a}(t) \= \frac{s-s^2-a+a^2 } {\bigl(t^2+(s-\frac12)^2\bigr)
\bigl(t^2+(a-\frac12)^2\bigr)}
\end{equation}
for $|\im t|<\re s-\frac12$. Moreover, the resolvent equation gives on
bounded functions in $C^\infty(\uhp)$ such that $\Dt f $ and
$\Dt^2 f$ are also bounded:
\begin{align}
L_{r_{s,a}} &(\Dt-s+s^2)(\Dt-a+a^2) f\= (R_s-R_a)
(\Dt-s+s^2)(\Dt-a+a^2) f
\\
\nonumber
&=(s-s^2-a+a^2)f\,.
\end{align}

Taking $s,a\in \CC^d$ with $\re s_j>1$, $\re a_j>1$, $s_j\neq a_j$ for
all $j$, we form
\[ k_{s,a} \bigl(u(z,w)\bigr) = \prod_j
r_{s_j,a_j}\bigl(u(z_j,w_j)\bigr)) \= \prod_j \biggl(
r_{s_j}\bigl(u(z_j,w_j)\bigr) -r_{a_j}\bigl(u(z_j,w_j)\biggr) \,,\]
and obtain a kernel operator $\Kcal_{s,a}$ on $\Gm\backslash \uhp^d$
given by the kernel function
\[R_{s,a}(z,w) \= \sum_{\gm\in \Gm} k_{s,a}(u(\gm z,w))\,.\]

We apply this operator to differentiable bounded functions $f$ on
$\Gm\backslash\uhp^d$ for which the derivatives
$\Dt_1^{b_1}\cdots\Dt_d^{b_d} f$ are bounded for all choices
$b_j\in \{0,1,2\}$.
\begin{align*}
f&\= \prod_j \frac1{(s_1-s_1^2-a_1+a_1^2)
\cdots(s_d-s_d^2-a_d+a_d^2)}\;\Kcal_{s,a} f_1\,,\\
f_1&\= (\Dt_1-s_1+s_1^2)
\cdots(\Dt_d-s_d+s_d^2)(\Dt_1-a_1+a_1^2)\cdots(\Dt_d-a_d+a_d^2) \;
f\,.
\end{align*}
Now we note that the values $\Kcal_{s,a}f(z)$ are given by a scalar
product in $L^2(\Gm\backslash\uhp^d)$:
\[ \Kcal_{s,a}f_1(z) \= \bigl\langle f_1,\overline{R_{s, a}
(z,\cdot)}\bigr\rangle\= \bigl\langle f_1,R_{\bar s,\bar a}
(z,\cdot)\bigr\rangle\,.\]
Taking the scalar product is a continuous operation
$L^2(\Gm\backslash\uhp^d)
\longrightarrow\CC$. Thus, using \eqref{Kdc} and \eqref{hsa} we
conclude that the $L^2$-expansion of $f_1$ is transformed in a
pointwise expansion:
\begin{align}\label{spe-Lf1}
\Kcal f_1(z) &\= \sum_\ell h_{ s, a}(t_\ell) \, \ps_\ell(z)\,
a_\ell^{(1)}\\
\nonumber
&\quad\hbox{} + \sum_\k 2c_\k \sum_{\mu\in \Lcal_\k} \int_0^\infty
\overline{h_{s,a}(t+\mu)} \, E(\k;it,i\mu;z)\, b_{\mu,\k}^{(1)}(t)\,
dt\,;\\
\nonumber
h_{s,a}(t) &\= \prod_j \frac{s_j-s_j^2-a_j+a_j^2}
{\bigl(t_j^2+(s_j-\frac12)^2\bigr)\bigl(t_j^2+(a_j-\frac12)^2\bigr)}\,,
\end{align}
where $a_\ell^{(1)} = \langle f_1, \ps_\ell\rangle$ and
$b_{\k,\mu}^{(1)}(t) =
\int_{\Gm\backslash\uhp^d} f_1(z) \overline{E(\k;it,i\mu;z)}\, d\mu(z)$.
Since $R_{s,a}(z,w)$ is bounded for $z\in \uhp^d$ uniformly in $z$ in
compact sets, the convergence of the expansion \eqref{spe-Lf1} is
also uniform on compact sets.

For two times differentiable functions in $L^2(\Gm\backslash\uhp^d)$
application of $\Dt_j$ changes $a_\ell$ in the spectral expansion
into $\bigl(\frac14-\mu_{\ell,j}^2\bigr)\, a_\ell$, and
$b_{\k,\mu}(t)$ into
$\bigl( \frac14+(t+\mu_j)^2\bigr)\, b_{\k,\mu}(t)$. Taking this into
account, the pointwise spectral expansion of $f$ takes the form
\begin{equation}\label{pse}
f(z) \= \sum_\ell \ps_\ell(z)\, a_\ell + \sum_\k 2c_\k \sum_{\mu\in
\Lcal_\k} \int_0^\infty E(\k;it,i\mu;z)\, b_{\mu,\k}(t)\, dt\,.
\end{equation}

Now we turn to the sum in \eqref{Kdef} defining $K(z,w)$. The sum is
locally finite in~$w$, uniform for $z$ in a fixed compact set, and
defines a smooth bounded differentiable function $f:w\mapsto K(z,w)$
with compact support modulo $\Gm$. Its derivatives are bounded,
uniform for $z$ in compact sets, hence its spectral expansion in~$w$
converges pointwise.

\subsection{Spectral measure}\label{sect-spm}
To prove Theorem~\ref{thm-spm} we use the estimate of the counting
function in Lemma~\ref{lem-ape}, and apply the mechanism of the
Selberg transform.

We take $k_j\in C_c^\infty[0,\infty)$ that satisfy $0\leq k_j\leq 1$,
$k_j=1$ on $[0,\eta_j]$ and $k_j=0$ on $[\dt_j,\infty)$ for
quantities $\eta=(\eta_j)_j$ and $\dt=(\dt_j)_j$ with
$0<\eta_j<\dt_j<\frac12$ to be chosen later. We form $k$ and $K$ as
in~\eqref{kdef} and~\eqref{Kdef}. We have seen in \eqref{K-est} that
$K(z,\cdot) \in L^2(\Gm\backslash\uhp^d)$.
We shall give two inequalities in which the norm $\|K(z,\cdot)\|_2$
occurs.

We have
\begin{align*}
\| K(z,\cdot)\|_2^2 &\= \int_{\Gm \backslash\uhp^d} |K(z,w)|^2\,
d\mu(w)\\
& \= \sum_{\gm,\dt\in \Gm} \int_{\fd} k(\gm z, w) \,k(\dt z, w)\,
d\mu(w)\\
&\= \sum_{\gm,\dt\in \Gm} \int_\fd k(\dt^{-1}\gm z,\dt^{-1} w)\,
k(z,\dt^{-1}w)\, d\mu(w)\\
&\= \sum_\gm \int_{\uhp^d} k(\gm z,w)\, k(z,w)\, d\mu(w)\,.
\end{align*}
The second factor restricts the domain of integration to $w$ with
$u(z_j,w_j) \leq \dt_j$ for all~$j$, and the first factor to $w$ with
$u\bigl((\gm z)_j,w_j\bigr) \leq \dt_j$ for all~$j$. For the
hyperbolic distances this means that
$\dist\bigl((\gm z)_j,z_j ) \leq 2\ups_j$, where $\ups_j$ corresponds
to $\dt_j$ according to the relation~\eqref{u-d}. For small values we
have $\dt_j \sim \frac14\ups_j^2$. Hence
$u(\gm_j z_j,z_j) \leq \tilde\dt_j$ with $\tilde\dt_j \sim 4\dt_j$ as
$\dt_j\downarrow0$. Hence, with $\tilde \dt=(\tilde\dt_j)_j$:
\begin{equation}\label{e1} \| K(z,\cdot)\|_2^2 \;\leq \;
\cnt(z;0,\tilde\dt) \int_{\uhp^d} k(z,w)\, d\mu(w)
\= \cnt(z;0,\tilde\dt) \prod_j h\bigl(\tfrac12\bigr)
\,.
\end{equation}
Note that Lemma~\ref{lem-se1}~a) implies that $\prod_j
h\bigl(\frac12\bigr)$ is a positive quantity between
$(4\pi)^d \allowbreak \prod_j(\dt_j-\nobreak\eta_j)$ and
$(4\pi)^d  \allowbreak \prod_j \dt_j$.

Let $X\in [1,\infty)^d$. We recall that in~\eqref{Ydef} we have given
a bounded subset of the spectral set depending on~$X$.
Theorem~\ref{thm-pspe} implies that
\begin{align}
\nonumber
\|K(z,\cdot)\|_2^2 &\= \sum_\ell |h(t_\ell)|^2 \,|\ps_\ell(z)|^2\\
\nonumber
&\qquad\hbox{} + \sum_\k 2c_\k \sum_{\mu\in \Lcal_\k} \int_0^\infty
|h(t+\mu)|^2 \, |E(\k;it,i\mu;z)|^2\, dt
\displaybreak[0]
\\
\nonumber
&\;\geq\; \min\{ |h(t)|^2\;:\; t\in Y(X)\} \; \biggl( \sum_{\ell\,,\;
t_\ell \in Y(X)} |\ps_\ell(z)|^2
\displaybreak[0]\\
\label{nest2}
&\qquad\hbox{}
+ \sum_\k 2c_\k \sum_{\mu\in \Lcal_\k} \int_{t\geq0\,,\, (t+\mu_j)_j
\in Y(X)} |E(\k;it,i\mu;z)|^2\, dt \biggr)\,.
\end{align}

To get a hold on a lower bound of $h$ on~$Y(X)$, we use
Lemma~\ref{lem-se1}~b). For $\tau\in Y(X)$ it gives
\begin{align*} \bigl| h_j(\tau_j)&-h_j\bigl(\frac12\bigr)\bigr|
\;\ll\; \dt_j^{3/2}\, X_j\,,\\
|h_j(\tau_j) |&\geq h_j\bigl(\frac12\bigr) - \oh \bigl( \dt_j^{3/2}
X_j \bigr)\= 4\pi(\dt_j-\eta_j) - \oh \bigl( \dt_j^{3/2} X_j \bigr)
\,.
\end{align*}
We take $\eta_j=\tfrac12\dt_j$, and $\dt_j = \e X_j^{-2}$ with $\e>0$
sufficiently small to have $|h_j(\tau_j)| \geq \dt_j$. This gives
\[ \min\bigl\{ |h(\tau)|^2\;:\; \tau\in Y(X) \bigr\} \;\geq \;
\frac1{(4\pi)^d}\, \prod_j h_j\bigl(\tfrac12\bigr)^2\,.\]
Thus we obtain from~\eqref{nest2} the inequality
\begin{equation}\label{e2}
\bigl\| K(z,\cdot)\bigr\|_2^2 \geq c\, S(X;z,z) \, \prod_j h_j\bigl(
\tfrac12\bigr)^2\,,
\end{equation}
for some positive constant~$c_1$, which does not depend on~$\Gm$. If
the $X_j$ are sufficiently large, the $\dt_j$ and the $\tilde\dt_j$
are sufficiently small to apply Lemma~\ref{lem-ape}. By \eqref{e2}
and~\eqref{e1} we get
\begin{align*}
S(X;z,z) & \leq \frac1{c_1\,\prod_j h_j\bigl(\tfrac12\bigr)^2} \,
\cnt(z;0,\tilde\dt)\, \prod_j h_j\bigl(\tfrac12\bigr)\\
&\;\ll_\Gm\; \frac1{\prod_j (\dt_j/2)}\, n(\tilde\dt^{-1/2},z)
\;\ll\; n(X,z)\,\prod_j X_j^2 \,.
\end{align*}


\newcommand\bibit[4]{
\bibitem {#1}#2, #3, #4}

\end{document}